\documentclass[12pt]{amsart}
\usepackage{amssymb,amsmath}
\usepackage{graphicx}
\usepackage{tikz}

\newtheorem{theorem}{Theorem}
\newtheorem{lemma}[theorem]{Lemma}
\newtheorem{proposition}[theorem]{Proposition}

\newtheorem{corollary}[theorem]{Corollary}

\theoremstyle{definition}
\newtheorem{definition}[theorem]{Definition}
\newtheorem{notation}[theorem]{Notation}

\allowdisplaybreaks

\makeatletter
\renewcommand\subsection{\@startsection {subsection}{1}{\z@}%
                                   {-1.5ex \@plus -1ex \@minus -.2ex}%
                                   {0.75ex \@plus.2ex}%
                                   {\normalfont\sc}}
\makeatother

\newcommand{\uu}{v}

\theoremstyle{remark}
\newtheorem{remark}[theorem]{Remark}

\def\myitemmarker{\P}
\newdimen\myindent
\setbox0=\hbox{\myitemmarker\ }
\myindent=\wd0
\graphicspath{{./Figures/}}

\newcommand{\ab}{\allowbreak}

\renewcommand{\thefootnote}{(\arabic{footnote})}

\newcommand{\thebottomline}{%
\renewcommand{\thefootnote}{}
\renewcommand{\footnoterule}{}
\phantom{M}\footnotetext{\hfill\tiny\textit{\noindent\romannumeral\day.\romannumeral\month.\romannumeral\year}}}

\newcommand{\cincludegraphics}[1]%
{\setbox1=\hbox{\includegraphics{#1}}%
$\vcenter{\hsize\wd1\box1}$}

\newcommand{\nc}{\mathit{nc}}

\newcommand\cA{{\mathcal A}}

\newcommand{\bC}{\mathbb{C}}

\newcommand{\bT}{\mathbb{T}}
\newcommand{\bZ}{\mathbb{Z}}
\newcommand{\diag}{\textrm{diag}}
\def\E{\mathrm{E}}

\def\Tr{\mathrm{Tr}}
\def\tr{\mathrm{tr}}

\newcommand{\Wg}{\operatorname{Wg}}
\def\cP{{\mathcal P}}

\let\phi=\varphi

\newcommand{\cat}{\mathit{Cat}}

\newcommand{\AN}{\mathcal{A}^{(p,q)}_{ \overrightarrow{\sigma}, \varepsilon, N }}

\newcommand{\tk}{\tilde\kappa}
\newcommand{\tphi}{\tilde\varphi}

\newcommand{\ith}{\mathit{th}}
\newcommand{\cD}{\mathcal{D}}

\renewcommand{\varepsilon}{\eta}

\newcommand{\VN}{\mathcal{V}_{\overrightarrow{\sigma}, \varepsilon, N }{(p,q)}}

\title[Freeness and the Transpose]{On the Partial Transpose
  \\[5pt] of a Haar Unitary Matrix}

\author[mingo]{James A. Mingo}
\address{Department of Mathematics and Statistics, Jeffery
  Hall, Queen's University, Kingston, Ontario, ONL K7L 3Z1,
  Canada}
\thanks{Research supported by a Discovery Grant from the
Natural Sciences and Engineering Research Council of Canada (JM); The Simons
  Foundation grant No. 360242 (MP); and  Narodowe Centrum Nauki grant 2016/23/D/ST1/01077 (KSz)}

\author[popa]{Mihai Popa}
\address{The University of Texas at San Antonio, Department
  of Mathematics, One UTSA Circle, San Antonio, Texas 78249,
  and \newline ${}\hspace{.35cm}{}$ Institute of Mathematics
  ``Simion Stoilow'' of the Romanian Academy, P.O. Box
  1-764, Bucharest, RO-70700, Romania}

\author[szpojankowski]{Kamil Szpojankowski}
\address{Wydzial Matehatyki i Nauk Informacyjnych,
  Politechnika Warszawska, ul. Koszykowa 75, 00-662
  Warszawa, Poland}

\begin{document}
\thispagestyle{empty}

\begin{abstract}
We consider the effect of a partial transpose on the limit $*$-distribution  of a Haar 
distributed random unitary matrix. If we fix the number of blocks, $b$, we show that the partial transpose can be decomposed along diagonals into a sum of $b$ matrices which are asymptotically free and identically  distributed. We then consider the joint effect of different block decompositions and show that under some mild assumptions we also get asymptotic freeness.
\end{abstract}

\maketitle

\section{Introduction}

Consider a sequence of random matrices $(U_N)_{N\geq 1}$,
where the matrix $U_N$ is an $N\times N$ Haar unitary
matrix. In our previous paper \cite{mpsz1} we studied
matrices obtained from Haar unitary matrices via a
permutation of the entries. More precisely for $\sigma$
being a permutation of the set $ \{ ( i, j): 1 \leq i, j
\leq N \}$, we considered a matrix $U_N$ defined as $[
  U_N^{\sigma}]_{i, j} = [ U_N ]_{ \sigma(i, j)} $. For a
sequence of matrices of growing sizes one considers a
sequence of permutations $(\sigma_N)_{N\geq 1}$, we
identified conditions which give that sequence
$U_N^{\sigma_N}$ is asymptotically circular and we also
found a conditions on pairs of sequences of permutations
$(\sigma_N)_{N\geq 1}$ and $(\sigma^\prime_N)_{N\geq 1}$
such that sequences $U_N^{\sigma_N}$ and
$U_N^{\sigma^\prime_N}$ are asymptotically free. Moreover we
showed that our conditions are satisfied with probability
one by sequences of random uniform permutations. Despite the
fact that conditions from \cite{mpsz1} cover many examples
of permutations, there are important examples of entry
permutations not covered by this result. One such example is
given by partial transposes, a significant class of entry
permutations of matrix entries which plays an important role
in quantum information theory. This paper is devoted to a
detailed study of asymptotic distribution of partial
transposes of Haar unitary matrices. 

\begin{notation} 

In this paper we shall consider the left and right partial
transpose of a block matrix. Let $A^T$ be the usual
transpose of a matrix $A$. If $N = b \times d$ we can write
$M_N(\bC) = M_b(\bC) \otimes M_d(\bC)$ and visualize this as
$M_b(M_d(\bC))$, i.e. $b \times b$ matrices with each entry
a $ d \times d$ matrix. As operators on the tensor product
we let $\Gamma = \mathit{id} \otimes T$ and $\Gamma^{(-1)} =
T \otimes \mathit{id}$. We shall write $\Gamma^{(1)}$ for
$\Gamma$ and usually also include the size of the blocks as
this determines the map. We call $\Gamma_{b, d}^{(-1)}$
the\textit{ left partial transpose} and $\Gamma_{b,d}^{(1)}$
the \textit{right partial transpose}.
	
If $ N $ is a positive integer we denote by $ [ N ] $ a
linearly ordered set with $ N $ elements. For simplicity, we
will identify $ [ N ]$ to the set $ \{1, 2, \dots, N
\}$. Suppose that $ N = b \cdot d $ and $ \vartheta \in \{
-1, 1 \} $. To bring our notation with entry permutations we
shall define the map $ \Gamma_{b, d}^{(\vartheta)} : [ N ]^2
\rightarrow [ N ]^2 $ as follows.  First, let $ \phi_{b, d}:
            [ N ]^2 \rightarrow ( [ b] \times [ d ])^2 $
            given by
\begin{align*}
		\phi_{b, d}( i, j) = ( a_1, a_2, a_{-1}, a_{-2})
\end{align*}
	whenever $ i = (a_1 -1) d + a_2 $ and $ j = (a_{-1} -1) d +
	a_{-2}.  $ The idea here is to have $(a_1, a_2, a_{-1},
	a_{-2})$ locate the $(a_2, a_{-2})$ entry of the $(a_1,
	a_{-1})$ block.  Then let $ \gamma_{b, d}^{(\vartheta)} : (
	[ b] \times [ d ])^2 \rightarrow ( [ b] \times [ d ])^2 $ be
	given by
	\[
	\gamma_{b, d}^{(\vartheta)} (a_1, a_2, a_{-1}, a_{-2} ) = ( a_{ \vartheta}, a_{-2\vartheta}, 
	a_{- \vartheta}, a_{2\vartheta}). 
	\]
	Thus $\gamma_{b,d}^{(1)}$ switches $a_2$ and $a_{-2}$,
	whereas $\gamma_{b,d}^{(-1)}$ switches $a_1$ and $a_{-1}$.
	Finally, put
	\[
	\Gamma_{b, d}^{ ( \vartheta)} = \phi_{b, d}^{ -1} 
	\circ \gamma_{b, d}^{(\vartheta)} \circ 
	\phi_{b, d}.
	\]
\end{notation}

In the case that $N=bd$ we can decompose the matrix $U_N$ as $b\times b$ block matrix with blocks of size $d\times d$ each. Hence we write $U_N=[U_{i,j}^{(d)}]_{1\leq i,j \leq b}$. It is natural then to consider the asymptotic join $\ast$--distribution of $\{U_{i,j}^{(d)}\}_{1\leq i,j \leq b}$ as $d\to\infty$, this question was settled in \cite{cu}, see discussion around equations \eqref{eq:matricial_cumulants_1}, \eqref{eq:matricial_cumulants_2} for details. Denote the resulting tuple of non--commutative random variables by $\{v_{i,j}\}_{1\leq i,j \leq b}$. In section 2 we present a parallel construction of a decomposition of a Haar unitary element $v$ and we explain that for $v_{i,j}$ as above the matrix $v=[v_{i,j}]_{1\leq i,j \leq b}$ where is indeed a Haar unitar and $v_{i,j}$ generate so called Brown algebra. Next we look at the transpose of $v$ that is $v^t=[v_{j,i}]_{1\leq i,j \leq b}$ and we present a decomposition of $v^t$ into a sum of $n$ operators which are free and $R$--diagonal, this allows us to find the distribution of $v^t$. Moreover we show that if we have $\mathcal{B}$ which is free from $\{v_{j,i}\}_{i,j}$ then $v$ is free from $M_n(\mathcal{B})$. In Section 3 we study asymptotic joint limiting distribution of different partial transposes of the same Haar unitary matrix. We find sufficient conditions under which partial transposes of the same Haar unitary matrix are asymptotically free.

\section{Limit distributions and freeness results \\[2pt]
in the Brown algebra}

\begin{definition}\label{def:brown_algebra}
The non-commutative unitary group $U_b^\nc$, also called the
\textit{Brown algebra}, is the universal $C^*$-algebra
generated by a unit and $b^2$ operators $\{ v_{ij}\}_{i,j =
  1}^b$ such that $ \sum_{k=1}^b v_{ik} v_{jk}^* =
\sum_{k=1}^b v_{ki}^* v_{kj} = \delta_{ij}$, i.e. that the
matrix $v = (\uu_{ij})_{i,j = 1}^d \in M_b(U_b^\nc)$ is
unitary.
\end{definition}

The algebra $U_b^\nc$ was constructed by Brown in \cite{b}
as the $(1, 1)$ corner algebra of the free product
$C^*$-algebra $M_b(\bC) * C(\bT)$ where $C(\bT)$ is the
$C^*$-algebra of continuous functions on the unit circle. We
can put a tracial state, $\phi_b$, on $M_b(\bC) * C(\bT)$
which is the free product of the normalized trace on
$M_b(\bC)$ and the state on $C(\bT)$ obtained from
integration with the normalized Haar measure on $\bT$ (see
McClanahan \cite[\S 3]{mc}). By virtue of the free product
construction we have that $M_b(\bC)$ and $C(\bT)$ are free
with respect to $\phi_b$ in $M_b(\bC) * C(\bT)$, (see
\cite[Prop. 1.5.5]{vdn}). Thus in $M_b(\bC) * C(\bT)$ we
have a Haar unitary $v \in C(\bT)$ (given by $v(z) = z$ for $z \in \bT$), which is $*$-free from the matrix
units $\{e_{ij}\}_{i,j = 1}^b$. We let $U_b^\nc =
e_{11}(M_b(\bC) * C(\bT))e_{11}$ and $v_{ij} = e_{1i} v
e_{j1} \in U_b^\nc$. Then $\{ v_{ij}\}_{i,j = 1}^b$ generate
$U_b^\nc$ and $M_b(U_b^\nc)$ can be identified with
$M_b(\bC) * C(\bT)$ (see
McClanahan \cite[Prop. 2.2]{mc}). We let $\phi: U_b^\nc \rightarrow \bC$
be given by $\phi(e_{11} f e_{11}) = b \phi_b(e_{11} f
e_{11})$ for $f \in C(\bT)$. Thus $\phi$ is a trace on
$U_b^\nc$.

The free $*$-cumulants of the Haar unitary $v$ are described
by saying that $v$ is $R$-diagonal and the non-vanishing
ones are given by the signed Catalan numbers $\beta_r =
(-1)^{r-1} \cat_{r-1}$ where $\cat_b$ is the $b^\ith$
Catalan number, see \cite[Cor. 15.1]{ns}. Using the fact
that the matrix units $\{ e_{ij} \}_{i,j = 1}^b$ are
$*$-free we may use the free compression result of Nica and
Speicher, \cite[Thm. 14.18]{ns}, to conclude that the free
$*$-cumulants (relative to $\phi$) of $\{v_{ij}\}_{i,j =
  1}^b$ are given by
\begin{gather}\label{eq:matricial_cumulants_1}
\kappa_{2r} \big(
v_{i_1, j_1},
v_{i_2, j_1}^\ast, 
v_{i_2, j_2},
v_{i_3, j_2}^\ast,
\dots,
v_{i_r, j_r}, 
v_{i_1, j_r}^\ast
\big)  = 
b^{1-2r}\beta_r\\
\label{eq:matricial_cumulants_2}
\kappa_{2r} \big(
v_{i_1, j_1}^\ast,
v_{i_1, j_2},
v_{i_2, j_2}^\ast,
v_{i_2, j_3}
\dots,
v_{i_r, j_r}^\ast,
v_{i_r, j_1}
\big)  = 
b^{1-2r}\beta_r 
\end{gather}
and all other free cumulants are zero. Since $\phi$ is
tracial each of (\ref{eq:matricial_cumulants_1}) and (\ref{eq:matricial_cumulants_2}) implies
the other (see \cite[Def. 15.3]{ns}).

Let $U_N$ be a $N \times N$ Haar distributed random unitary
matrix. It was shown by Voiculescu that $U_N$ is
asymptotically $*$-free from constant matrices (see
\cite[Thm. 23.14]{ns} for a proof using the Weingarten
calculus).  Suppose that $ N = bd $, then we can see $ U_N $
as a $ N \times N $ block matrix with block entries $ U_{i,
  j}$ for $ 1\leq i, j \leq n $. Then, for fixed $b$, the
joint distribution of the block matrices $\{ U_{ij}
\}_{i,j=1}^n$ converges (as $d \rightarrow \infty$) to the
joint distribution of $\{ v_{ij} \}_{i,j=1}^b$ which is
given by (\ref{eq:matricial_cumulants_1}) and (\ref{eq:matricial_cumulants_2}). This was already
observed by C\'ebron and Ulrich in \cite[Cor. 2.8 and
  Thm. 3.3]{cu}.
  
\subsection{Diagonal Decomposition}
Given a unital $*$-algebra $\cA$ and a matrix $a = (a_{ij})_{i,j = 1
}^b \in M_b(\cA)$ we let
\[
a^t
  =\left[ \begin{matrix} 
           a_{11} & a_{21} & \cdots & a_{b1} \\
           a_{12} & a_{22} & \cdots & a_{b2} \\
           \vdots & \vdots &        & \vdots \\
           a_{1b} & a_{2b}  & \cdots & a_{bb}
           \end{matrix}\right].
\] 
We call $a^t$ the \textit{transpose} of $a$. Since the
entries come from a non-commutative algebra we no longer
expect to have $(ab)^t = b^t a^t$. In this sub-section we
shall show how to decompose $v^t$ into $b$ pieces each
$*$-free from each other. Let
\[
           s = 
\left[\begin{matrix}
0      &    1   &    0   & \cdots &   0    \\
0      &    0   &    1   & \cdots &   0    \\
\vdots & \vdots & \vdots & \ddots & \vdots \\
\vdots & \vdots & \vdots &        &   1    \\
1      &    0   & \vdots & \cdots &   0  
\end{matrix}\right].
\]
Then $s = (s_{ij})_{ij}$ where $s_{ij} = \delta_{i+1, j}$
(all indices mod $n$), then and $s^i = 1$ for $i \equiv 0$
$\pmod b$. If $\phi: \cA \rightarrow \bC$ is a linear map
with $\phi(1) = 1$, we let $\Phi : M_b(\cA) \rightarrow
\bC$ given by $\Phi( a ) =\phi (a_{11} + \cdots +
a_{nn})/n$. If $I_b$ denotes the identity matrix of
$M_b(\cA)$ then $\Phi(I_b) = 1$. If $\phi$ is a trace on
$\cA$ then $\Phi$ is a trace on $M_b(\cA)$ and $\phi(s^i)
= 0$ for $i \not \equiv 0$.

\subsection{The $*$-distribution of $v^t$}

Now let $\cA = U_b^\nc$ and let us return to our Haar
unitary $v \in M_b(U_b^\nc)$.

\begin{definition}
\label{def:diagonal_decomposition}
Let $w_{1,0} = \diag( v_{11},\ab \dots,\ab v_{bb})$ be the
$b \times b$ diagonal matrix with diagonal entries $ v_{11},
\dots, v_{bb}$. For $1 \leq k \leq b-1$ let $w_{1, k}$ be
the diagonal matrix $\diag( v_{k+1,1}, v_{k+2,2}, \ab\dots,
v_{k,b} )$. Then
\[
v^t = w_{1, 0} s^0 + w_{1,1} s^1 + \cdots + w_{1, b-1}
s^{b-1}
\]
is the \textit{diagonal decomposition} of $v^t$. For $0 \leq
k \leq b-1$, let $v_k = w_{1,k} s^k$.
\end{definition}

Then $v^t = v_0 + \cdots + v_{b-1}$ and we shall show in
Theorem \ref{thm:freeness_over_C} that the family $\{ v_0,
v_1, \dots, v_{b-1}\}$ is $*$-free. When $b = 3$ our
decomposition looks like
\[
v_0 = \left[\begin{matrix} 
v_{11} & 0 & 0 \\ 0 & v_{22} & 0 \\ 0 & 0 & v_{33}\\
\end{matrix}\right], 
v_1 = \left[\begin{matrix} 
0 & v_{21} & 0 \\ 0 & 0 & v_{32} \\ v_{13} & 0 & 0\\
\end{matrix}\right], 
v_2 = \left[\begin{matrix} 
0 & 0 & v_{31} \\ v_{12} & 0 & 0 \\ 0 & v_{23} & 0\\
\end{matrix}\right].
\]
To demonstrate $*$-freeness we shall restate the results in
Eq's (\ref{eq:matricial_cumulants_1}) and
(\ref{eq:matricial_cumulants_2}) using the symmetric group.
  
For a positive integer $m$ let $[m] = \{1, 2, \dots m\}$ and
$S_m$ denote the permutation group of $[m]$. Moreover we let
$[\pm m] = \{1, -1, 2, -2, \dots, m, \ab -m\}$ and $S_{\pm
  m}$ denote the permutation group of $[\pm m]$. We shall
regard $S_m$ as the subgroup of $S_{\pm m}$ of permutations
acting trivially on $\{-1, -2, \dots, -m\}$. Thus for $\pi
\in S_m$ and $k \in [m]$ we have $\pi(-k) = -k$. Let $\delta
\in S_{\pm m}$ be the permutation with cycle decomposition
$(1,-1)(2, \ab-2)\cdots (m, -m)$. Let $\gamma$ be the
permutation in $S_m$ with cycle decomposition $(1, 2, 3,
\dots, m)$. Following our convention, $\gamma \delta
\gamma^{-1}$ has the cycle decomposition $(1, -m)(-1, 2),
(-2, 3) \cdots (-(m-1), m)$. Given $\epsilon = (\epsilon_1,
\dots, \epsilon_m) \in \bZ_2^m$ with $\bZ_2 = \{-1, 1\}$, we
consider $\epsilon$ to also be a permutation in $S_{[\pm
    m]}$ by setting $\epsilon(k) = \epsilon_{|k|} k$ for $k
\in [\pm m]$. See Remark \ref{rem:non_vanishing_cumulant}
for some illustrations of our notation. For a sequence
$i_{\pm 1}, i_{\pm 2}, \dots, i_{\pm m} \in [n]$ we let
$\ker(i)$ be the partition of $[\pm m]$ such that $i_r =
i_s$ if and only if $r$ and $s$ are in the same block of
$\ker(i)$. Since we must deal with $*$-moments we need a way
to record all possible mixed $*$-moments of the $u_{ij}$. To
this end we will let $a^{(1)} = a$ and $a^{(-1)} = a^*$ for
any element, $a$, of a $*$-algebra. The restatement of
Eq. (\ref{eq:matricial_cumulants_1}) and
Eq. (\ref{eq:matricial_cumulants_2}) now becomes
Eq. (\ref{eq:entrywise_cumulants}) below.
  
\begin{lemma}\label{lemma:cumulant_value}
Let $\epsilon_1, \dots, \epsilon_m \in \bZ = \{-1, 1\}$ and
$i_{\pm 1}, \dots, i_{\pm m} \in [b]$. Then
\begin{equation}\label{eq:entrywise_cumulants}
\kappa_m \big( v_{i_1i_{-1}}^{(\epsilon_1)}, \dots,
v_{i_mi_{-m}}^{(\epsilon_m)}\big) = 0
\end{equation}
unless: $(i)$ $m$ is even; $(ii)$ $\epsilon_k +
\epsilon_{k+1} = 0$ for $1 \leq k \leq m-1$; $(iii)$
$\ker(i) \geq \epsilon \gamma \delta
\gamma^{-1}\epsilon$. When these conditions are satisfied
the cumulant is $b^{1-m} \beta_{m/2}$.
\end{lemma}

\begin{figure}[t]
\begin{center}
\begin{tikzpicture}
\node  at (1,1) {1};
\node  at (2,1) {$-1$};
\node  at (3,1) {2};
\node  at (4,1) {$-2$};
\node  at (5,1) {3};
\node  at (6,1) {$-3$};
\node  at (7,1) {4};
\node  at (8,1) {$-4$};
\node  at (9,1) {$\cdots$};
\node  at (10,1) {$m$};
\node  at (11, 1) {$-m$};
\draw  (1, 0.75) -- (1, 0) -- (10,0) -- (10, 0.75);
\draw (2,0.75) -- (2,0.5) -- (4,0.5) -- (4, 0.75);
\draw (3, 0.75) -- (3, 0.25) -- (5, 0.25) -- (5, 0.75);
\draw (6,0.75) -- (6,0.5) -- (8,0.5) -- (8,0.75);
\draw (7,0.75) -- (7,0.25) -- (8.25,0.25);
\draw (9.5,0.5) -- (11,0.5) -- (11, 0.75); 
\end{tikzpicture}
\end{center}
\caption{\small One of the two possible pairings that gives
  a non-crossing cumulant in Remark
  \ref{rem:non_vanishing_cumulant}. Shown is $(1,
  m)(-1,-2)\ab(2, 3)(-3, -4) \cdots \ab (-(m-1),
  m)$.\label{fig:first_pairing}}
\end{figure}
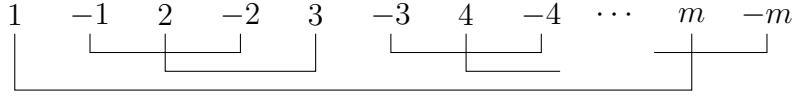

\begin{remark}\label{rem:non_vanishing_cumulant}
As there are only two possible $\epsilon$'s which produce a
non-zero cumulant there are only two possible values for
$\epsilon \gamma \delta \gamma^{-1}\epsilon$. When $\epsilon
= (1, -1, \dots, 1, -1)$ we have $\epsilon \gamma \delta
\gamma^{-1}\epsilon = (1, m)(-1,-2)(2, 3)(-3, -4) \cdots \ab
(-(m-1), m)$. When $\epsilon = (-1, 1, \dots, -1, 1)$ we
have $\epsilon \gamma \delta \gamma^{-1}\epsilon = (-1,\ab
-m) (1,2) (-2, \ab-3) \ab (3, 4) \cdots \ab (-(m-2), -(m-1))
\ab (m-1, m)$. See Figures \ref{fig:first_pairing} and
\ref{fig:second_pairing} for illustrations. The condition in
Lemma \ref{lemma:cumulant_value} becomes either, 
\[
i_1 = i_m, i_{-1} = i_{-2}, i_2 = i_3, i_{-3} = i_{-4}, \dots, 
i_{-m} = i_{-(m-1)}
\]
or
\[
i_1 = i_2, i_{-1} = i_{-m}, i_{-2} = i_{-3}, i_{3} = i_{4}, \dots, 
i_{m} = i_{m-1}.
\]
\end{remark}

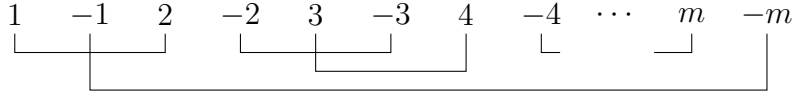
\begin{figure}[t]
\begin{center}
\begin{tikzpicture}
\node  at (1,1) {1};
\node  at (2,1) {$-1$};
\node  at (3,1) {2};
\node  at (4,1) {$-2$};
\node  at (5,1) {3};
\node  at (6,1) {$-3$};
\node  at (7,1) {4};
\node  at (8,1) {$-4$};
\node  at (9,1) {$\cdots$};
\node  at (10,1) {$m$};
\node  at (11, 1) {$-m$};
\draw  (1, 0.75) -- (1, 0.5) -- (3,0.5) -- (3, 0.75);
\draw (2,0.75) -- (2,0) -- (11,0) -- (11, 0.75);
\draw (4, 0.75) -- (4, 0.5) -- (6, 0.5) -- (6, 0.75);
\draw (5,0.75) -- (5,0.25) -- (7,0.25) -- (7,0.75);
\draw (8,0.75) -- (8,0.5) -- (8.25,0.5);
\draw (9.5,0.5) -- (10,0.5) -- (10, 0.75); 
\end{tikzpicture}
\end{center}
\caption{\small One of the two possible pairings that gives
  a non-crossing cumulant in Remark
  \ref{rem:non_vanishing_cumulant}. Shown is $(-1, -m)
  (1,2)\ab (-2, \ab-3) \ab (3, 4) \cdots \ab (-(m-2),
  -(m-1)) \ab (m-1, m)$.\label{fig:second_pairing}}
\end{figure}

\begin{notation}
For a matrix $a = (a_{ij})_{i,j = 1}^b \in M_b(\cA)$, let
$a^{[1]}_{ij} = a_{ij}$ and $a^{[-1]}_{ij} = a_{ji}^*$. With
this notation $a^{[-1]}_{ij}$ is the $(i,j)$-entry of
$a^*$. Note that with our other notation we have
$a^{(-1)}_{ij} = a_{ij}^*$.
\end{notation}

\begin{remark}\label{rem:non_vanishing_cumulant_bis}
With the notation above we have that equation
(\ref{eq:entrywise_cumulants}) becomes
\begin{equation}\label{eq:entrywise_cumulants_revised}
\kappa_m \big( v_{i_1i_{-1}}^{[\epsilon_1]}, \dots,
v_{i_mi_{-m}}^{[\epsilon_m]}\big) = 0
\end{equation}
unless: $(i)$ $m$ is even; $(ii)$ $\epsilon_k =
-\epsilon_{k+1}$ for $1 \leq k < m$; $(iii)$ $\ker(i) \geq
\gamma \delta \gamma^{-1}$, i.e.
\[
i_1 = i_{-m}, i_{2} = i_{-1}, i_{3} = i_{-2}, \dots,
i_{m} = i_{-(m-1)}.
\]
Notice that a matrix $a =
(a_{ij})_{ij}$ is $R$-cyclic exactly when all cumulants
$\kappa_b(a_{i_1i_{-1}}, a_{i_2i_{-2}}, \dots,
a_{i_bi_{-b}})$ vanish except possibly when $(i)$ $b$ is
even and $(ii)$ $\ker(i) \geq \gamma \delta \gamma^{-1}$. So
condition (\ref{eq:entrywise_cumulants_revised}) is a combination of
$R$-cyclicity and $R$-diagonality; which might be called
$R^{\,*}\!$-\textit{cyclicity}.
\end{remark}

\begin{notation}
Recall that $w_{1,k} = \diag( v_{k+1,1}, v_{k+2,2}, \dots,
v_{k,n} )$. We shall interpret the indices of $v$ modulo
$n$; so that when $b = 5$, $v_{-2,4} = v_{3,4}$.  With this
convention we let
\[
w_{1,k,j} = s^j w_{1, k} s^{-j} = \diag( v_{k+j +1,j+ 1},
v_{k+j+2,j+2}, \dots, v_{k+j,j} ).
\]
Then we have $w_{1,k,j}^* = \diag(v_{k+j +1,j+ 1}^*,
v_{k+j+2,j+2}^*, \dots, v_{k+j,j}^*)$. For $k \geq 0$, let
\[
w_{-1,k} = \diag\Big( v_{b-k +1, 1}^{[-1]},
v_{b-k+2,2}^{[-1]}, \dots, v_{b-k,b}^{[-1]} \Big)
\]
\[
= \diag( v_{1, b-k +1}^{*}, v_{2, b-k+2}^{*}, \dots, v_{b,
  b-k}^{*}) = s^{-k} w_{1, k}^* s^k = w^*_{1,k,-k}.
\]
For example when $b=5$ and $k = 2$ we have
\[
w_{1, 2} = \diag( v_{31}, v_{42}, v_{53}, v_{14}, u_{25})
\]
and
\[
w_{-1,2} = \diag( v_{41}^{[-1]}, v_{52}^{[-1]},
v_{13}^{[-1]}, v_{24}^{[-1]}, v_{35}^{[-1]})
\]
\[
= \diag( v_{14}^*, v_{25}^*, v_{31}^*, v_{42}^*, v_{53}^*) =
s^{-2} w_{1,2}^* s^2.
\]
The idea is that putting the minus sign in $w_{-1,k}$ suggests 
replacing $k$ by $-k$. When $k = 0$ we have
\[
w_{1,0} = \diag(v_{11}, \dots, v_{nn}) \mathrm{\ and\ }
w_{-1,0} = w_{1,0}^*.
\]
Thus for $\epsilon \in \{-1, 1\}$ we can rewrite our
definition of $w$ for $k \geq 0$ as
\[
w_{\epsilon, k} = \diag( v^{[\epsilon]}_{\epsilon \cdot k+1,
  1}, \dots, v^{[\epsilon]}_{\epsilon \cdot k,b}).
\]
As before we let $w_{-1,k,j} = s^j w_{-1, k} s^{-j}$. Then
for $k \geq 0$
\[
w_{-1, k,j} = \diag( v_{k+j +1,j+ 1}^{[-1]},
v_{k+j+2,j+2}^{[-1]}, \dots, v_{k+j,j}^{[-1]} )
\]
\[
=\diag\Big(v_{j+1, k+j+1}^*, v_{j+2, j+k + 2}^*, \dots,
v_{j, k+j}^* \Big).
\]
\end{notation}
Then for $\epsilon \in \{-1, 1\}$, $k \geq 0$, and $j \in
\bZ$ we have
\[
w_{\epsilon, k, j} = \diag( v^{[\epsilon]}_{\epsilon \cdot k
  + j + 1, j + 1}, v^{[\epsilon]}_{\epsilon \cdot k + j + 2,
  j + 2}, \cdots, v^{[\epsilon]}_{\epsilon \cdot k + j, j}).
\]

Recall that $v_k = w_k s^k$, for $k > 0$. Thus

\[
v_k^* = s^{-k} w_{1,k}^* = s^{-k} w_{1,k}^* s^k s^{-k}
= w_{-1, k} s^{-k}. 
\]
Hence $v_k^{(\epsilon)} = w_{\epsilon,k} s^{\epsilon \cdot
  k}$. The $m^{\ith}$ entry of $w_{\epsilon,k,j}$ is
$v^{[\epsilon]}_{\epsilon \cdot k + j + m, j + m}$.

\begin{notation}
Let $\cD \subseteq M_b(U_b^\nc)$ denote the subalgebra of $b
\times b$ diagonal scalar matrices. We let $\tilde \phi$
denote the conditional expectation from $M_b(U_b^\nc)$ to
$\cD$ given by
\[
\tilde\phi( (a_{ij})_{ij}) = \diag( \phi(a_{11}), \dots, \phi(a_{bb})). 
\]
We let $\tilde \kappa_m$ denote the $m^\ith$ $\cD$-valued
cumulant. Thus for $a_1, \dots, a_m \ab\in M_b(U_b^\nc)$ we
have
\[
\tilde\kappa_m(a_1, \dots, a_m) = \sum_{\pi \in NC(m)} 
\mu(\pi, 1_m) \tilde\phi_\pi(a_1, \dots, a_m).
\]
If $a_1, \dots, a_m$ are diagonal matrices in
$M_b(U_b^\nc)$, as are our matrices $w_{\epsilon, k, j}$,
then the $\cD$-valued cumulants can just be computed entry
wise.
\end{notation}

\begin{proposition}\label{prop:vanishing_cumulant}
The $l^\ith$ entry of the diagonal matrix 
\[
\tilde \kappa_m( w_{\epsilon_1, i_1}, w_{\epsilon_2, i_2, \epsilon_1 i_1}, 
\cdots,
w_{\epsilon_m, i_m, \epsilon_1 i_1 + \cdots + \epsilon_{m-1} i_{m-1}})
\]
is
\[
\kappa_m( v^{[\epsilon_1]}_{j_1 + l, l}, v^{[\epsilon_2]}_{j_1 + j_2 + l, j_1 + l}, \dots, 
v^{[\epsilon_m]}_{j_1 + \cdots + j_m + l, j_1 + \cdots + j_{m-1} + l})
\]
where $j_k = \epsilon_k i_k$. 
\end{proposition}

\begin{proof}
This follows the previous observation that the $m^{\ith}$
entry of $w_{\epsilon,k,j}$ is $v^{[\epsilon]}_{\epsilon
  \cdot k + j + m, j + m}$.
\end{proof}

Recall that we shall write $i \equiv j$ to mean equivalence
modulo $b$.
\begin{corollary}\label{cor:vanishing_cumulants}
\[
\tilde \kappa_m( w_{\epsilon_1, i_1}, w_{\epsilon_2, i_2,
  \epsilon_1 i_1}, \cdots, w_{\epsilon_m, i_m, \epsilon_1
  i_1 + \cdots + \epsilon_{m-1} i_{m-1}}) \not= 0
\]
only if: $(i)$ $m$ is even; $(ii)$ $\epsilon_k i_k +
\epsilon_{k+1} i_{k+1} \equiv 0$, for $1 \leq k < m$;
$(iii)$ $i_k \equiv i_{k+1}$ for $1 \leq k < m$.
\end{corollary}

\begin{proof}
For a diagonal valued cumulant to be non-zero there has to
be a non-zero entry. By Lemma \ref{lemma:cumulant_value} and
Proposition \ref{prop:vanishing_cumulant} this can only
happen when $(i)$ $m$ is even; $(ii)$ $\epsilon_k +
\epsilon_{k+1} = 0$ for $1 \leq k < m-1$; and $(iii)$
$\epsilon_k i_k + \epsilon_{k+1} i_{k+1} \equiv 0$ for $1
\leq k < m$. By $(ii)$ this last condition is equivalent to
$i_k \equiv i_{k+1}$ for $1 \leq k < m$.
\end{proof}

\begin{lemma}\label{lemma:preliminary}
Let $a = (a_{ij})_{ij} \in M_b(U_b^\nc)$. We let $a = c_0s^0
+ \cdots + c_{b-1}s^{b-1}$ be the diagonal decomposition of
$a$ where $c_0, \dots, c_{b-1}$ are diagonal matrices and we
let $c_{i,j} = s^jc_is^{-j}$.
\begin{enumerate}
\item
Let $\pi$ be in $NC(l)$. If for each block $V = (j_1, \dots,
j_r)$ of $\pi$ we have $i_{j_1} + \cdots + i_{j_r} \equiv 0
$ then
\[
\tphi_\pi(c_{i_1}s^{i_1}, \dots, c_{i_l}s^{i_l}) 
= 
\tphi_\pi(c_{i_1,0}, \dots, c_{i_l i+1 + \cdots + i_{l-1}}).
\]

\item
For any $\pi \in NC(l)$ we have
\begin{multline*}
\tphi_\pi(c_{i_1}s^{i_1}, \dots, c_{i_l}s^{i_l}) 
= 
\tphi_\pi(c_{i_1,0}, \dots, c_{i_l, i+1 + \cdots + i_{l-1}}) \\
\times \tphi_\pi(s^{i_1}, \dots, s^{i_l}).
\end{multline*}
\end{enumerate}

\end{lemma}

\begin{proof}
As we have seen 
\[
c_{i_1}s^{i_1} \cdots c_{i_l}s^{i_l} = 
c_{i_1,0}c_{i_2i_1} \cdots c_{i_l, i_1 + \cdots +i_{l-1}} s^{i_1 +
\cdots + i_l}.
\]
$c_{i_1,0}c_{i_2i_{1}} \cdots c_{i_l, i_1 + \cdots + i_{l-1}}$
is a diagonal matrix and $s^{i_1 + \cdots + i_l}$ will be
$0$ on the diagonal unless $i_1 + \cdots + i_l \equiv
0$. Thus $\tphi(c_{i_1}s^{i_1} \cdots c_{i_l}s^{i_l}) = 0$
unless $i_1 + \cdots + i_l \equiv 0$.

To prove $(i)$ note that for
$\tphi_\pi(c_{i_1}s^{i_1}, \dots, c_{i_l}s^{i_l}) \not = 0$
we must have that for each block $V = (j_1, \dots, j_r)$ of
$\pi$ we have $i_{j_1} + \cdots + i_{j_r} \equiv 0$ and thus
\[
\tphi(c_{i_{j_1}}s^{i_{j_1}} \cdots c_{i_{j_r}}s^{i_{j_r}})
= \tphi(c_{i_{j_1}} \cdots c_{i_{j_r}, i_{j_1}+ \cdots + i_{j_{r-1}}})
  \tphi(s^{i_{j_1} + \cdots + i_{j_r}}). 
\]
Since this applies for every block we have 
\[
\tphi_\pi(c_{i_1}s^{i_1}, \dots, c_{i_l}s^{i_l}) 
= 
\tphi_\pi(c_{i_1}, \dots, c_{i_l, i_{1} + \cdots + i_{l-1}})
\tphi_\pi(s^{i_{j_1} + \cdots + i_{j_r}}).
\]

To prove $(ii)$ note that $\tphi_\pi(s^{i_1}, \dots,
s^{i_l}) \in \{0, 1\}$ with $\tphi_\pi(s^{i_1}, \dots,
s^{i_l})\ab = 1$ only when for each block $V = (j_1, \dots,
j_r)$ of $\pi$ we have $i_{j_1} + \cdots + i_{j_r} \equiv
0$. Thus both sides of the equation in claim $(ii)$ vanish unless
the hypothesis in $(i)$ applies, in which case $(ii)$ follows from $(i)$.
\end{proof}

\begin{theorem}\label{thm:freeness_over_D}
$v_0, v_1, \dots, v_{b-1}$ are $*$-free over $\cD$.
\end{theorem}

\begin{proof}
We shall show that for $d_1, \dots, d_{m-1} \in \cD$ we have
$\tilde\kappa_m( v_{i_1}^{(\epsilon_1)}d_1,\ab \dots,
v_{i_{m-1}}^{(\epsilon_{m-1})}  d_{m-1},
v_{i_m}^{(\epsilon_m)}) = 0$ unless: $(i)$ $m$ is even;
$(ii)$ $i_1 \equiv i_2 \equiv \cdots \equiv i_m$; $(iii)$
and $\epsilon_k + \epsilon_{k+1} \equiv 0$.  We shall let
$d_{i,j} = s^j d_i s^{-j}$ and $\tilde d = d_{1,i_1} \cdots
d_{b-1, i_1 + \cdots + i_{b-1}}$. Recall that
$v_{i_k}^{(\epsilon_k)} = w_{\epsilon_k, i_k} s^{\epsilon_k
  i_k}$.

\begin{eqnarray}\label{eq:vanishing_cumulants}\lefteqn{
\tk_m(w_{\epsilon_1,i_1}s^{\epsilon_1 i_1}d_1, \dots,
w_{\epsilon_{m-1}, i_{m-1}}s^{\epsilon_{m-1}i_{m-1}}d_{m-1},
w_{\epsilon_m,i_m} s^{\epsilon_m i_m}) } \notag\\
& = &
\sum_{\pi \in NC(m)} \kern-0.75em\mu(\pi, 1_m) \tphi_\pi
(w_{\epsilon_1,i_1}s^{\epsilon_1 i_1}d_1, \dots, \cdots
d_{m-1}, w_{\epsilon_m,i_m} s^{\epsilon_m i_m}) \notag\\
& \stackrel{(*_1)}{=} &
\sum_{\pi \in NC(m)} \mu(\pi, 1_m) \tphi_\pi
(w_{\epsilon_1,i_1}, \dots, w_{\epsilon_m,i_m, \epsilon_1
  i_1 + \cdots + \epsilon_{m-1}i_{m-1}}) \notag\\
&&
\qquad\qquad \mbox{} \times \tphi_\pi(s^{\epsilon_1 \cdot
  i_1}, \dots, s^{\epsilon_m \cdot i_m}) \, \tilde d
\notag\\
& = &
\sum_{\pi \in NC(m)} \kern-1em \mu(\pi, 1_m) \kern-0.5em
\mathop{\sum_{\sigma \in NC(m)}}_{\sigma \leq \pi} \kern-0.5em
\tphi_\pi(s^{\epsilon_1 \cdot i_1}, \dots, s^{\epsilon_m \cdot i_m})
 \, \tilde d  \notag \\
 &&
\qquad\qquad \mbox{} \times \tk_\sigma (w_{\epsilon_1,i_1},
\dots, w_{\epsilon_m,i_m, \epsilon_1 i_1 + \cdots +
  \epsilon_{m-1}i_{m-1}}) \notag\\
& \stackrel{(*_2)}{=} & \kern-1.5em
\sum_{\pi \in NC(m)} \kern-0.75em \mu(\pi, 1_m) \kern-1em
\mathop{\sum_{\sigma \in NC(b)}}_{\sigma \leq \pi}
\kern-0.75em \tk_\sigma (w_{\epsilon_1,i_1}, \dots,
w_{\epsilon_m,i_m, \epsilon_1 i_1 + \cdots +
  \epsilon_{m-1}i_{m-1}}) \, \tilde d \notag \\
& = & \kern-1em
\sum_{\pi \in NC(m)} \mu(\pi, 1_m)
\tphi_\pi(w_{\epsilon_1,i_1}, \dots, w_{\epsilon_m,i_m,
  \epsilon_1 i_1 + \cdots + \epsilon_{m-1}i_{m-1}}) \,
\tilde d \notag\\
& = &
\tk_m(w_{\epsilon_1,i_1}, \dots, w_{\epsilon_m,i_m,
  \epsilon_1 i_1 + \cdots + \epsilon_{m-1}i_{m-1}}) \,
\tilde d.
\end{eqnarray}
In the calculation above $(*_1)$ holds by Lemma
\ref{lemma:preliminary} $(ii)$.  By Corollary
\ref{cor:vanishing_cumulants} we know that
$\tk_\sigma(w_{\epsilon_1,i_1}, \dots, w_{\epsilon_m,i_m,
  \epsilon_1 i_1 + \cdots + \epsilon_{m-1}i_{m-1}}) = 0$
unless for each block $(l_1, \dots, l_t)$ of $\sigma$ we
have $i_{l_1} + \cdots + i_{l_t} \equiv 0$. If $\sigma \leq
\pi$ the same condition holds for any block of $\pi$. Thus
for these $\sigma$'s we have $\tphi_\sigma(s^{\epsilon_1
  \cdot i_1}, \dots, s^{\epsilon_m \cdot i_m})
=\tphi_\pi(s^{\epsilon_1 \cdot i_1}, \dots, s^{\epsilon_m
  \cdot i_m}) = 1$. This justifies $(*_2)$.

By Corollary \ref{cor:vanishing_cumulants} the cumulant in
\eqref{eq:vanishing_cumulants} vanishes unless: $(i)$ $m$ is
even; $(ii)$ $\epsilon_k + \epsilon_{k+1} = 0$ for $1 \leq k
< n$; and $(iii)$ $i_1 \equiv \cdots \equiv i_m$.
\end{proof}

\begin{remark}
From \cite[Thm. 3.19]{am} we can conclude that $v_0$,
$\{v_1, v_{b-1}\}$, \dots, $\{v_{b/2-1}, v_{b/2 + 1}\}$,
$v_{b/2}$ are free (assuming $b$ is even). Because of the
additional structure of $v$ we get the following stronger
conclusion.
\end{remark}

\begin{theorem}\label{thm:freeness_over_C}
$v_0, v_1, \dots, v_{b-1}$ are $R$-diagonal and $*$-free over $\bC$.
\end{theorem}

\begin{proof}
In Theorem \ref{thm:freeness_over_D} we proved freeness over
$\cD$, the diagonal scalar matrices, so it suffices to show that $\tk_m(v_{i_1}^{(\epsilon_1)}, \dots, v_{i_m}^{(\epsilon_m)}) = \kappa_m( v_{i_1}^{(\epsilon_1)}, \dots,\ab v_{i_m}^{(\epsilon_m)})$. By Lemma
\ref{lemma:cumulant_value}, $\tk_m(v_{i_1}^{(\epsilon_1)}, \dots, v_{i_m}^{(\epsilon_m)})$, the diagonal valued cumulant in
\eqref{eq:vanishing_cumulants}, is actually a multiple of the
identity matrix. On the other hand $\phi = \tr \circ \phi_b$, where $\phi_b$ is the state on $M_b(\bC) \otimes U_b^{nc}$ given in Definition \ref{def:brown_algebra}. Thus $\kappa_m = \tr \circ \tk_m$. Hence $\tk_m(v_{i_1}^{(\epsilon_1)}, \dots, v_{i_m}^{(\epsilon_m)}) = \kappa_m( v_{i_1}^{(\epsilon_1)}, \dots,\ab v_{i_m}^{(\epsilon_m)})$. Thus we have vanishing of mixed cumulants and
hence freeness.
\end{proof}

\begin{theorem}\label{thm:16}
The transpose of a Haar unitary has the same
$*$-distri\-bution as $b^{-1}(u_1 + \cdots + u_{b^2})$ where
$u_1, \dots, u_{b^b}$ are $b^2$ $*$-free Haar unitaries;
i.e., $v^t \stackrel{D}{\sim} b^{-1}v^{\boxplus b^2}$.
\end{theorem}

\begin{proof}
From Theorem \ref{thm:freeness_over_C} we only have to show
that each $bv_i$ has same $*$- distribution as the sum of
$b$ $*$-free Haar unitaries. We already have shown that
$v_i$ is $R$-diagonal, so it remains to show that
$\kappa_{2m}(bv_i, bv_i^*, \dots, \ab bv_i, bv_i^*) = b
\beta_m$. By equation (\ref{eq:vanishing_cumulants}) in the
proof of Theorem \ref{thm:freeness_over_C} we have
$\tilde\kappa_{2m}(v_i, v_i^*, \dots, \ab v_i, v_i^*) =
\tilde\kappa_{2m}( w_{1,i}, w_{-1, i, 1}, \dots, w_{1,
  i},\ab w_{-1,i,i})$. By Proposition
\ref{prop:vanishing_cumulant}, the $l^{th}$ entry of this
diagonal matrix is $\kappa_{2m}(v_{i+l,l}, v_{i+l,l}^*,
\dots, v_{i+l,l},\ab v_{i+l,l}^*) = b^{1 - 2m} \beta_m$, with
the last equality by Equation
(\ref{eq:matricial_cumulants_1}). Thus
$\tilde\kappa_{2m}(v_i,\ab v_i^*, \dots, \ab v_i, \ab v_i^*) = b^{1
  - 2m} \beta_m$ and hence $\kappa_{2m}(v_i, \ab v_i^*,
\dots, \ab v_i, v_i^*) = b^{1 - 2m} \beta_m$ as claimed.
\end{proof}

\subsection{Free Independence results in the Brown algebra}
  
The result below is a non-commutative analogue of the asymptotic freeness between Haar unitaries and independent random matrices.

   Suppose that 
   $ ( \mathcal{A}, \varphi ) $
   is a 
   $ \ast $-non-commutative probability space
   such that
   $ \mathcal{A} $
   is a unital 
   $ \ast$-algebra containing the 
   $ \ast$-algebra generated by
   $ \{ v_{i, j}: \  1 \leq i,j \leq n \} $
   and some 
   $\ast$-algebra
    $ \mathcal{B}$
    free from the family
   $ \{ v_{i, j}: \  1 \leq i,j \leq n \} $. 
   As before, denote by 
   $ v$ 
   the 
   $ n \times n $ 
   matrix whose 
   $ (i, j)$-entry is 
   $ {v}_{i,j}$ (in particular,
    $ v $ 
    is unitary ) and denote by
    $ v^t $ 
    the matrix transpose of 
    $ v $,
    i.~e. the 
    $(i, j)$-th entry of 
    $ {v}^t $
    is
  $ {v}_{j, i} $.

 \begin{theorem}\label{thm:free:brown}

  With the notations from above,
  $ v $ is free from  $ M_n(\mathcal{B} )$ 
  with respect to
  $ \Phi = \varphi \circ \tr $.
 \end{theorem}
 
 \begin{proof}
 
 It suffices to show that 
 \begin{align*}
 \Phi
 \big(
 v^{ ( \varepsilon_1)} 
 A_1
 v^{( \varepsilon_2 ) } &
 A_2 
 \cdots
 v^{ ( \varepsilon_m )}
 A_m
  \big) \\
  = & \sum_{ \pi \in NC(m)} 
  \kappa^{\Phi}_{\pi}
  [v^{ ( \varepsilon_1)},
  v^{( \varepsilon_2 ) } ,
   \dots, v^{ ( \varepsilon_m )} ]
   \cdot
   \Phi_{ \textrm{Kr}(\pi)} 
   [ A_1, \dots, A_m ]
  \end{align*}
  for any  positive integer 
  $ m $, any matrices 
  $ A_1, A_2, \dots, A_m \in M_n ( \mathcal{B}) $
  and any
  $ \varepsilon_1, \dots, \varepsilon_m \in
  \{ 1, \ast \} $. 
 
 With the notations
 \[
 \varepsilon_s (i, j) = \left\{ 
 \begin{array}{ll}
 	(i, j) & \textrm{ if } \varepsilon_s =1 \\
 	(j, i)  & \textrm{ if } \varepsilon_s = \ast
 \end{array}\right.
 \hspace{.5cm}
 \textrm{ and }
 v^{( \varepsilon_s )} = \left\{ 
 \begin{array}{ll}
 	v & \textrm{ if } \varepsilon_s =1 \\
 	v^*  & \textrm{ if } \varepsilon_s = \ast
 \end{array}\right.
 ,
 \]
we have $
 [ v^{(\varepsilon)}]_{i, j} = {v}_{ \varepsilon(i,j)}^{(\varepsilon)} 
 $,
 thus we get
 \begin{align*}
 \Phi
 \big(
 v^{ ( \varepsilon_1)} 
 A_1 &
 v^{( \varepsilon_2 ) } 
 A_2 
 \cdots
 v^{ ( \varepsilon_m )}
 A_m
  \big) \\
  = &
  \sum_{ \overrightarrow{i}, \overrightarrow{j}}
  \varphi
  \big(
  v_{ \varepsilon_1( i_1, j_1)}^{ ( \varepsilon_1)}
  a^{(1)}_{j_1, i_2}
  v_{ \varepsilon_2( i_2, j_2)}^{ ( \varepsilon_2)}
  \cdots
  a^{(m-1)}_{j_{m-1}, i_m}
  v_{ \varepsilon_m( i_m, j_m)}^{ ( \varepsilon_m)}
   a^{(m)}_{j_{m}, i_1}
   \big)\\
   =
   \sum_{ \pi \in NC(m) }&
   \sum_{ \overrightarrow{i}, \overrightarrow{j}}
   \kappa^\varphi_{\pi}
   [
   v_{ \varepsilon_1( i_1, j_1)}^{ ( \varepsilon_1)}
 ,
    \cdots
  v_{ \varepsilon_m( i_m, j_m)}^{ ( \varepsilon_m)}
   ]
   \cdot
   \phi_{ \textrm{Kr}(\pi)} 
   [
   a^{(1)}_{j_1, i_2},
   \dots,
    a^{(m)}_{j_{m}, i_1}
   ],
  \end{align*}
 where 
 $  \overrightarrow{i} = (i_1, \dots, i_m ) $,
 $ \overrightarrow{j} = ( j_1, \dots, j_m ) $
 and
 $ a_{i,j}^{(s)} $ is the 
 $(i, j)$-th entry of the matrix
 $ A_s $.
 
 It suffices to show
   that the equality below holds true for any
   non-crossing partition 
   $ \pi $:
   \begin{align}\label{eq:pi}
   \kappa^\Phi_{\pi}
    [v^{ ( \varepsilon_1)}, &
    v^{( \varepsilon_2 ) } , 
     \dots,
      v^{ ( \varepsilon_m )} ]
     \cdot 
     \Phi_{ \textrm{Kr}(\pi)} 
     [ A_1, \dots, A_m ]\\
     = &
       \sum_{ \overrightarrow{i}, \overrightarrow{j}}
       \kappa^\varphi_{\pi}
       [
 {v}_{ \varepsilon_1( i_1, j_1)}^{ ( \varepsilon_1)}
     ,
        \cdots
{v}_{ \varepsilon_m( i_m, j_m)}^{ ( \varepsilon_m)}
       ]
       \cdot
       \phi_{ \textrm{Kr}(\pi)} 
       [
       a^{(1)}_{j_1, i_2},
       \dots,
        a^{(m)}_{j_{m}, i_1}
       ].\nonumber
   \end{align}
   
  Let
  $ \pi \in NC(m) $.
 One of the blocks of 
  $ \pi $,
   say
   $ B = ( t+1, \dots, s ) $,
    is an interval.
   Then, denoting by
    $ \widetilde{\pi} $
    the non-crossing partition obtained by removing the block
    $ B $ 
    from 
    $ \pi $,
    we have that
    \begin{align*}
    &\kappa^\Phi_{\pi}
       [v^{ ( \varepsilon_1)}, 
       v^{( \varepsilon_2 ) } , 
        \dots,
         v^{ ( \varepsilon_m )} ]
         = K_1 \cdot 
         \kappa^\Phi_{ \widetilde{\pi}} 
         [
    v^{ ( \varepsilon_1)}, 
    \dots
          v^{( \varepsilon_t ) } ,
          v^{( \varepsilon_{s+1} ) } , 
           \dots,
            v^{ ( \varepsilon_m )}     
         ]\\
   &\Phi_{ \textrm{Kr}(\pi)} 
     [ A_1, \dots, A_m ]
     = F_1
      \cdot 
    \Phi_{\textrm{Kr}(\widetilde{\pi}) } 
      [ A_1, 
      \dots,
      A_{t-1},
      (A_{t}\cdot A_{s}),
      A_{s+1},
       \dots,
        A_m ]  
           \end{align*}
    where
    \begin{align*}
    K_1 & = 
    \kappa^\Phi_{ s -t }
    \big( 
    v^{( \varepsilon_{t+1})},
    v^{( \varepsilon_{t+2})},
    \dots,
    v^{( \varepsilon_{s})}
    \big)\\
    F_1 & =
    \Phi 
    (A_{t+1}
    )
    \cdot
    \Phi 
       (A_{t+2}
       )
       \cdots
       \Phi 
          (A_{s-1}
          ).
    \end{align*}
    Also, we have that
    \begin{align*}
   \kappa^{\varphi}_{\pi}
       [ &
 {v}_{ \varepsilon_1( i_1, j_1)}^{ ( \varepsilon_1)}
     , 
        \dots
{v}_{ \varepsilon_m( i_m, j_m)}^{ ( \varepsilon_m)}
       ]\\
       & = 
        K_2
         (
        \overrightarrow{i_{\alpha}}, 
        \overrightarrow{ j_{\alpha} }
         )
        \cdot
     \kappa^{\varphi}_{ \widetilde{\pi} }
     [
 {v}_{ \varepsilon_1( i_1, j_1)}^{ ( \varepsilon_1)}
         ,
         \dots,
 {v}_{ \varepsilon_1( i_t, j_t)}^{ ( \varepsilon_t)}
             ,
{v}_{ \varepsilon_1( i_{s+1}, j_{ s + 1})}^{ ( \varepsilon_{ s+1})}
        , 
        \dots,
{v}_{ \varepsilon_m( i_m, j_m)}^{ ( \varepsilon_m)}          
     ]
    \end{align*}
    and
    \begin{align*}
    & \phi_{ \textrm{Kr}(\pi)} 
          [ 
          a^{(1)}_{j_1, i_2},
          \dots,
           a^{(m)}_{j_{m}, i_1}
          ]\\
           & =
     F_2
      (
            \overrightarrow{i_{\alpha}}, 
            \overrightarrow{ j_{\alpha} }
             )
     \cdot 
       \varphi_{ \textrm{Kr}(\widetilde{\pi})}   
       [
       a^{(1)}_{j_1, i_2},
                \dots,
                a^{(t-1)}_{j_{t-1}, i_t},
                \big(
    a^{(t)}_{j_t, i_{t+1}}  
    \cdot
    a^{(s)}_{j_s, i_{s+1}}            
                \big),
      a^{(s+1)}_{j_{s+1}, i_{s+2}}, \\
      & \hskip24em\dots,
      a^{(m)}_{j_{m}, i_1}          
       ] 
    \end{align*}
    where
    $ \overrightarrow{i_{\alpha}}
    =
    ( i_{t+1}, i_{t+2}, \dots, i_{s})
    $ ,
    respectively
    $ \overrightarrow{j_{\alpha}}
       =
       ( j_{t+1}, j_{t+2}, \dots, j_{s})
       $, 
       and
   \begin{align*}
   K_2
   (
              \overrightarrow{i_{\alpha}}, 
              \overrightarrow{ j_{\alpha} }
               )
    = &
      \kappa^{\varphi}_{ s-t }
      ( 
  u_{ \varepsilon_{t+1}( i_{ t+1}, j_{t+1})}^{ ( \varepsilon_{t+1})} ,
   u_{ \varepsilon_{t+2}( i_{ t+2}, j_{t+2})}^{ ( \varepsilon_{t+2})} , 
   \dots, 
   u_{ \varepsilon_{s}( i_{ s}, j_{s})}^{ ( \varepsilon_{s})} 
     ) \\
     F_2
     (
                \overrightarrow{i_{\alpha}}, 
                \overrightarrow{ j_{\alpha} }
                 )
      = & 
      \phi( a^{(t+1)}_{ j_{t+1}, i_{t+2}})
        \cdot
        \phi( a^{(t+2)}_{ j_{t+2}, i_{t+3}})
        \cdots
        \phi( a^{(s-1)}_{ j_{s-1}, i_{s}}).
     \end{align*}
   
   Since 
   $ v $ 
   is R-diagonal,
   $ K_1 $
   cancels unless  
      $ \varepsilon_{t+l} \neq \varepsilon_{t+l+1} $
       for all
       $ l=1, \dots, s-t-1$;
   according to        
   (\ref{eq:matricial_cumulants_1}) and (\ref{eq:matricial_cumulants_2}),
   so does 
   $ K_2 $ 
   for any
   $ \overrightarrow{i} $, 
   $ \overrightarrow{j} $.
  
  Suppose that 
    $ \varepsilon_{t+l} \neq \varepsilon_{t+l+1} $
          for all
          $ l=1, \dots, s-t-1$.
          Then
    \begin{align*}
    K_1 = (-1)^{ s-t -1} \cdot \mathit{Cat}_{s-t-1}.
    \end{align*}      
 On the other hand, utilizing 
  (\ref{eq:matricial_cumulants_1}) and (\ref{eq:matricial_cumulants_2}),
  we have that
  \begin{align*}
   K_2 
  (
             \overrightarrow{i_{\alpha}}, 
             \overrightarrow{ j_{\alpha} }
              )
              = 
  n^{1-2r}(-1)^{r -1} \mathit{Cat}_{r-1}
  \cdot
  \delta_{j_s}^{i_{t+1}}
  \delta_{j_{t+1}  }^{ i_{t+2}} 
  \delta_{j_{t+2}  }^{ i_{t+3}}         
  \cdots
  \delta_{j_{ s -1 }}^{ i_s }
   \end{align*}
  therefore
  \begin{align*}
  \sum_{ \overrightarrow{i_{\alpha}}, 
              \overrightarrow{ j_{\alpha} }
              }
              &
    \varphi(a^{(t)}_{j_t, i_{t+1}}  
          \cdot
          a^{(s)}_{j_s, i_{s+1}} )
      K_2 
     (
                \overrightarrow{i_{\alpha}}, 
                \overrightarrow{ j_{\alpha} }
                 )
                 \cdot
     F_2 
    (
               \overrightarrow{i_{\alpha}}, 
               \overrightarrow{ j_{\alpha} }
                )  \\
     &  =   \left[ \sum_{ \overrightarrow{i_{\alpha}}, 
                   \overrightarrow{ j_{\alpha} }
                   }
     \varphi([ A_t A_s ]_{ j_t, i_{s+1}})
                         \cdot
    n^{1-2(s-t)}(-1)^{s-t -1} \mathit{Cat}_{s-t-1} 
     \right] \\
    & 
     \cdot  
    \prod_{l=t+1}^{s-1}
    \Tr\circ \varphi(A_{l})
    \cdot \big| \big\{  ( \overrightarrow{i_{\alpha}}, 
                  \overrightarrow{ j_{\alpha} } ) 
                  :\
                   \delta_{j_s}^{i_{t+1}}
                    \delta_{j_{t+1}  }^{ i_{t+2}} 
                    \delta_{j_{t+2}  }^{ i_{t+3}}         
                    \cdots
                    \delta_{j_{ s -1 }}^{ i_s } =1
                   \big\} \big|\\
    & =\varphi(  
       [ A_t A_s ]_{ j_t, i_{s+1}})
                            \cdot  
    K_1
    \cdot
   F_1                        
  \end{align*}
  
  Denoting
  $ \overrightarrow{i_{\beta}} = 
  (i_1, i_2, \dots, i_{t}, i_{s+1}, i_{s + 2},
  \dots, i_m )$
  and, respectively\\
  $ \overrightarrow{j_{\beta}} = 
  (
   j_1, j_2, \dots, j_{t}, j_{s+1}, j_{s + 2},
   \dots, j_m )$,
   the relation above gives that
     \begin{align*}
        & \sum_{ \overrightarrow{i}, \overrightarrow{j}}
         \kappa^{\varphi}_{\pi} 
         [
v_{ \varepsilon_1( i_1, j_1)}^{ ( \varepsilon_1)}
       ,
          \cdots
v_{ \varepsilon_m( i_m, j_m)}^{ ( \varepsilon_m)}
         ]
         \cdot
         \phi_{ \textrm{Kr}(\pi)} 
         [
         a^{(1)}_{j_1, i_2},
         \dots,
          a^{(m)}_{j_{m}, i_1}
         ]
     \\
       =  \mbox{} & 
       K_1 \cdot F_1 \cdot
       \sum_{
       \overrightarrow{ i_{\beta} }, 
       \overrightarrow{ j_{\beta} }
       }
       \Big(
         \kappa^{\varphi}_{ \widetilde{\pi} }
           [
v_{ \varepsilon_1( i_1, j_1)}^{ ( \varepsilon_1)}
               ,
               \dots,
 v_{ \varepsilon_1( i_t, j_t)}^{ ( \varepsilon_t)}
                   ,
v_{ \varepsilon_1( i_{s+1}, j_{ s + 1})}^{ ( \varepsilon_{ s+1})}
              , 
              \dots,
v_{ \varepsilon_m( i_m, j_m)}^{ ( \varepsilon_m)}          
           ] \\
            & \cdot
       \varphi_{ \textrm{Kr}(\widetilde{\pi})}   
           [
           a^{(1)}_{j_1, i_2},
                    \dots,
                    a^{(t-1)}_{j_{t-1}, i_t},
                    [A_t A_s]_{ j_t, i_{s+1}},
          a^{(s+1)}_{j_{s+1}, i_{s+2}},
          \dots,
          a^{(m)}_{j_{m}, i_1}          
           ]        
            \big)
     \end{align*}
   and 
   (\ref{eq:pi})
    follows by induction on 
   $ m $.
  \end{proof}
  
  \begin{remark}\label{non-unit-inv}
 $ M_n(\mathcal{B})$
  is \emph{not} free from
 $ v^t$. 
 For example, 
 $ v^t $ 
 is not free from the matrix
 \[
 A = 
 \left[\begin{array}{c|c}
 \begin{array}{cc}0&1\\ 1&0\end{array} & 0 \\ \hline
 0 & 0_{ n -2} 
 \end{array}\right].
 \]
 If 
 $ v^t $
 and
 $ A $
 were free, since
  $ \Phi (A) = \Phi(v) = 0 $,
  we would have that
  \[
  \Phi\big( 
   v^t A ( v^t )^\ast A v^t A ( v^t )^\ast A 
    \big)=0.
  \]
  On the other hand, denoting by
  $ a_{i, j} $
  the  $(i, j)$-th entry of 
  $ A $, 
  we have that
  \begin{align*}
 \Phi\big(
 v^t A ( v^t )^\ast  A v^t  A & ( v^t )^\ast  A 
 \big) \\
= & \frac{1}{b}
 \sum_{ i_1, \dots, i_8 =1}^b
 \varphi \big( {v}_{i_2, i_1}
  a_{i_2, i_3} 
 {v}^\ast_{i_3, i_4} 
 a_{i_4, i_5}
 {v}_{i_6, i_5}
  a_{i_6, i_7} 
{v}^\ast_{i_7, i_8}
 a_{i_8 , i_1}
 \big), 
  \end{align*}
 where
 $ a_{i, j} $ 
 is the $(i, j)$-th entry of 
 $ A $. 
 In particular, 
 $ a_{i,j}= 0 $
 whenever
 $ i> 2 $
 or 
 $ j > 2 $,
 so
 \begin{align*}
 \Phi\big(
 v^t A  & ( v^t )^\ast  A   v^t  A  ( v^t )^\ast  A 
 \big) \\
= & \frac{1}{n}
 \sum_{ i_1, \dots, i_8 =1}^2
 \varphi \big(
  v_{i_2, i_1}
  v^\ast_{i_3, i_4} 
v_{i_6, i_5}
  v^\ast_{i_7, i_8}
 \big)
 a_{i_2, i_3} a_{i_4, i_5} 
 a_{i_6, i_7} a_{i_8 , i_1}\\
 = &
 \frac{1}{n}
 \sum_{ \pi \in NC(4) }
 \sum_{ i_1, \dots, i_8 =1}^2
 \kappa^{\varphi}_{ \pi }[
  v_{ i_2, i_1},
 v^\ast_{ i_3, i_4}, 
 v_{i_6, i_5},
 v^\ast_{i_7, i_8}  ]
  a_{i_2, i_3} a_{i_4, i_5} 
  a_{i_6, i_7} a_{i_8 , i_1}
  \end{align*} 
  Since
  $ \varphi( v_{i, j} ) = 0 $
  for all
  $ i, j $,
  there are only three partitions to consider:
  $ \{ (1, 2), (3, 4) \} $,
  $ \{ (1, 4), (2, 3) \} $,
  and
  $ \{ ( 1, 2, 3, 4)\} $.
  
  If
  $ \pi =\{ (1, 2), (3, 4) \} $,
  then
  \begin{align*}
  \kappa_{ \pi } 
  [
    v_{ i_2, i_1},
   v^\ast_{ i_3, i_4}, 
   v_{i_6, i_5},
   v^\ast_{i_7, i_8}  ]
   =\kappa_2 (v_{ i_2, i_1},
      v^\ast_{ i_3, i_4} )
       \cdot 
       \kappa_2 ( v_{i_6, i_5},
          v^\ast_{i_7, i_8} ).
\end{align*}
From Equation (\ref{eq:matricial_cumulants_1}), we have that
$ \kappa_2 (v_{ i_2, i_1},
      v^\ast_{ i_3, i_4} )
      \neq 0 
      $
      only if
      $ i_2 = i_3 $.
Then 
$ a_{i_2, i_3} =0 $
since
$ a_{i, i}= 0 $
for all 
$ i $.

Similarly, if
$ \pi = \{ ( 1, 4), (2, 3)\}$, 
then
 \begin{align*}
  \kappa_{ \pi } 
  [
    v_{ i_2, i_1},
   v^\ast_{ i_3, i_4}, 
   v_{i_6, i_5},
   v^\ast_{i_7, i_8}  ]
   =\kappa_2 (v_{ i_2, i_1},
      v^\ast_{ i_7, i_8} )
       \cdot 
       \kappa_2 ( v^\ast_{i_3, i_4},
          v_{i_6, i_5} ),
\end{align*}
and now
$ \kappa_2 ( v^\ast_{i_3, i_4},
          v_{i_6, i_5} )
           \neq 0$
only if
$ i_4 = i_5 $,
that is
$ a_{i_4, i_5} = 0 $.

Finally, if
$ \pi = \{ (1, 2, ,3 , 4)\} $,
then
  $ \kappa_{ \pi } 
    [
      v_{ i_2, i_1},
     v^\ast_{ i_3, i_4}, 
     v_{i_6, i_5},
     v^\ast_{i_7, i_8}  ]
     \neq 0 
     $
 only if
 $i_1=i_4 $,
 $ i_3 = i_6 $,
 $ i_5 = i_8 $
 and
 $ i_2 = i_7 $.
  \begin{align*}
  \Phi\big(
 v^t A   ( v^t )^\ast  A   v^t  A  ( v^t )^\ast  A 
  \big)
  =  & - n^{-4} \sum_{ i_1, i_2, i_3, i_5 =1}^2
  a_{i_2, i_3} a_{i_3, i_2} a_{i_1, i_5} a_{i_5, i_1} \\
  = & -4n^{-4} \neq 0. \hskip10em \square 
  \end{align*} 
   \end{remark}

 \section{Asymptotic free independence of different partial transposes of a Haar unitary random matrix}

  \subsection{Framework and previous results on partial transposes}
 
  In \cite{mp3}, we gave a necessary and sufficient condition for the asymptotic free independence of different families of partial transposes of Wishart random matrices. 
  More precisely, suppose that 
  $ ( b_N)_N$ ,$ ( b^\prime_N)_N $, $ (d_N)_N $
  and
  $ ( d^\prime_N )_N $
  are non-decreasing sequences of positive integers
 such that for each $ N $ we have 
  $ b_N \cdot d_N = b^\prime_N \cdot d^\prime_N = M_N $
  and that
  $ \displaystyle \lim_{N \rightarrow \infty} M_N = \infty $.
  If 
  $ W_N $
   denotes a 
   $ M_N \times M_N $ 
   Wishart random matrix, and
   $ \vartheta, \vartheta^\prime \in \{1, -1\} $, then the asymptotic free independence of the families
   $ \big( W_N^{ \Gamma_{b_N, d_N}^{( \vartheta)}} \big)_N $
   and
   $ \big( W_N^{ \Gamma_{b_N^\prime, d_N^\prime}^{( \vartheta)}} \big)_N $
   is equivalent to the condition
           \begin{align}\label{condi:19}
  \lim_{N \rightarrow \infty }
  \frac{1}{M_N^2}
         \big| \big\{  (i, j) \in [ M_N ]^2:
           \Gamma_{b^\prime_N, d^\prime_N}^{ (\vartheta^\prime)}
           (i, j) =   \Gamma_{b_N, d_N}^{ (\vartheta)}(i, j) &
            \big\} \big|= 0. 
           \end{align}
 
The main results of this Section, Theorems \ref{thm:24} and \ref{thm:25}, is that condition (\ref{condi:19}) is equivalent to asymptotic $\ast$--freeness for different partial transposes also in the case of Haar unitary random matrices.

  We will use the following technical results (proved in \cite[Theorem 3.2]{mp3}) on partial transposes.
  
  \begin{lemma}\label{lemma:gamma}${}$
  
 \begin{enumerate}
 \item[$(i)$] If
  $ \vartheta = \vartheta^\prime$, 
  then condition \emph{(\ref{condi:19}) }
    is equivalent to 
    \[  \lim_{N \rightarrow\infty }\frac{l.c.m.(b_N, b^\prime_N)}{\min(b_N,b_N')}
    =
    \lim_{N \rightarrow\infty }
    \frac{ l.c.m. (d_N, d^\prime_N)}{\min(d_N, d^\prime_N)}
    = \infty.
    \]
      
\noindent Also, for each $ N $,  the sets
  \begin{align*}
   \big\{  (i, j) \in [ M_N ]^2: \ 
 ( \Gamma_{b_N, d_N}^{ (\vartheta)})^{ -1}
 \circ
 \Gamma_{b^\prime_N, d^\prime_N}^{ (\vartheta^\prime)}
 (i, j) = (i, j)
  \big\} 
 \end{align*}
 and (for 
 $ s =1, 2 $, where $\pi_s$ is the projection on the $i$-th coordinate),
  \begin{align*}
   \big\{  (i_1, i_2, j) \in [ M_N ]^3 : \ 
 \pi_s \circ  \Gamma_{b_N, d_N}^{ (\vartheta)} (i_1, j)
 =\pi_s
 \circ
 \Gamma_{b^\prime_N, d^\prime_N}^{ (\vartheta^\prime)}
 (i_2, j) 
  \big\} 
 \end{align*}
 have the same number of elements.
 
 \item[$(ii)$]  If
  $ \vartheta \neq \vartheta^\prime $
  then the condition \emph{(\ref{condi:19})} is equivalent to
  \[ \lim_{N \rightarrow \infty }
  b_N \cdot d^\prime_N = 
  \lim_{N\rightarrow \infty}
  b^\prime_N \cdot d_N 
   =\infty.
   \]
  
  \noindent Moreover, in this case, condition \emph{(\ref{condi:19})} 
  implies that (for $ s =1, 2 $)
  \begin{enumerate}
  \item[(a)]$ | \big\{  (i_1, i_2, j) \in [ M_N ]^3 : \ 
    \Gamma_{b_N, d_N}^{ (\vartheta)} (i_1, j)
  =
  \Gamma_{b^\prime_N, d^\prime_N}^{ (\vartheta^\prime)}
  (i_2, j) 
   \big\}| \ab = o(M^2)$
  \item[(b)] $ |
   \big\{  (i_1, i_2, j) \in [ M_N ]^3 : \ 
  \pi_s \circ  \Gamma_{b_N, d_N}^{ (\vartheta)} (i_1, j)
  =\pi_s
  \circ
  \Gamma_{b^\prime_N, d^\prime_N}^{ (\vartheta^\prime)}
  (i_2, j) 
   \big\} | \ab
   = o(M^3) $
  \end{enumerate}
 \end{enumerate}
  \end{lemma} 

Note that for $\alpha \not = \beta$ and $b_N \sim N^\alpha$ and $b_N' \sim N^\beta$ then 
\[
\lim_{N \rightarrow \infty} \frac{l.c.m.(b_N, b_N')}{\min(b_N, b_N')} = \infty
\]

 Before stating and proving the main theorems of this section, we need to review some results on the unitary Weingarten function.
 
 \subsection{On the unitary Weingarten calculus}

 \begin{notation}
 Recall the unitary Weingarten function 
 $\Wg$.
  It is a central element of the group ring $\bC[S_n]$
   and, by definition, it is the inverse of the function 
   $\sigma \mapsto N^{\#(\sigma)}$
    where 
    $N \geq n$ and 
     $\#(\sigma)$ 
     is the number of cycles in the cycle decomposition of 
     $\sigma$. 
 
 Collins  \cite{co} showed that for
  $U = (u_{ij})_{ij}$
   a
    $N \times N$ 
Haar distributed random unitary matrix we have
 \begin{equation}\label{eq:weingarten_symmetric}
 \E(u_{i_1j_1} \cdots u_{i_nj_n}
 \overline{u_{i'_1j'_1}} \cdots \overline{u_{i'_nj'_n}}) =
 \sum_{\pi, \sigma \in S_n} \Wg_N(\sigma^{-1}\pi) 
 \delta_{i, i' \circ \sigma} \delta_{j, j' \circ \pi}
 \end{equation}
 where 
 $\delta_{i, i'\circ \sigma} = 1$
  when
   $i_1 = i'_{\sigma(1)}$, \dots, $i_n = i'_{\sigma(n)}$ 
   and
    $0$ 
    otherwise. 
 \end{notation} 
  
Collins also showed that for
 $\sigma \in S_n$
  we have
\begin{equation}\label{eq:weingarten_order}
\Wg_N(\sigma)
 = w_1(\sigma) N^{-2n + \#(\sigma)} + O(N^{-2n + \#(\sigma)-2} )
\end{equation}
and provided a explicit function for
 $w_1(\sigma)$.
 Indeed
  $w_1(\sigma)$
is the product 
   $\prod_{i=1}^k (-1)^{l_i-1} C_{l_i-1}$ 
where we decompose
    $\sigma$
into a product of cycles
  $c_1 \cdots c_k$ 
and 
  $l_i$
is the number of elements in the 
   $i^{th}$
cycle, and 
    $C_r$
is the
 $r^{th}$
  Catalan number
   $ \displaystyle \frac{1}{r+1} \binom{2r}{r}$.
   
   As shown in \cite[Prop. 10]{mp2}, we can rewrite equation (\ref{eq:weingarten_symmetric}) using pairings. To do so we must recall a lemma about pairings from \cite[Lemma 2]{mp1}. 
   
   \begin{lemma}
   Given two pairings $p$ and $q$ of $[n]$ we consider $p$ and $q$ to be permutations and then decompose the product $pq$ into cycles. When we do this we can write $pq = c_1 c'_1 c_2 c'_2 \cdots c_k c'_k$ with $c'_l = q c_l^{-1}q$. Moreover the blocks of $p \vee q$ are 
       $\{ c_1 \cup c'_1, \dots, c_k \cup c'_k \}$. Thus $2\#(p \vee q) = \#(pq)$. 
   \end{lemma} 

\begin{notation}\label{not:21}
Given pairings
 $p$ 
and 
 $q$
of 
  $[n]$
we denote by 
   $\Wg_N(p, q)$ 
the value 
    $\Wg_N(\sigma)$
where 
 $\sigma \in S_{n/2}$
has the same cycle decomposition as
  $c_1 c_2 \cdots c_k$,
where 
$pq = c_1 c'_1 c_2 c'_2 \cdots c_k c'_k$. 
With this notation, Collins' formula (\ref{eq:weingarten_order}) becomes for 
$p, q \in \cP_2(n)$
\begin{equation}\label{eq:weingarten_order_pairing}
\Wg_N(p,q) = w_1(\sigma) N^{-n + \#(pq)/2} + O(N^{-n + \#(pq)/2-2} ).
\end{equation}	
\end{notation}

Next, we will remind the main result from \cite{mpsz1} concerning the asymptotic behavior 
(as 
$ N \rightarrow \infty $)
of
\begin{align*}
\E \circ \tr \big(
( U^{ \sigma_{1}})^{ \varepsilon_1} 
( U^{ \sigma_{2}} )^{ \varepsilon_2}
 \cdots
(  U^{ \sigma_{n} } )^{ \varepsilon_n}
\big)
\end{align*}
where
$ U $ 
is a 
$ N \times N $
Haar unitary random matrix, 
$ \sigma_1, \dots, \sigma_n $
are entry permutations of 
$ N \times N $ 
matrices, and
$ \varepsilon: [ n ] \rightarrow \{ 1, \ast \}$.

With the notations
\[
\varepsilon_s (i, j) = \left\{ 
\begin{array}{ll}
(i, j) & \textrm{ if } \varepsilon_s =1 \\
(j, i)  & \textrm{ if } \varepsilon_s = \ast
\end{array}\right.
\hspace{.5cm}
\textrm{ and }
 z^{( \varepsilon_s )} = \left\{ 
 \begin{array}{ll}
 z & \textrm{ if } \varepsilon_s =1 \\
 \overline{z}  & \textrm{ if } \varepsilon_s = \ast
 \end{array}\right.
 ,
\]
we have that
 \begin{align*}
 \E(\tr 
 \big( ( U^{ \sigma_1})^{ \varepsilon_1} &
( U^{ \sigma_2})^{ \varepsilon_2}
\cdots 
( U^{ \sigma_n})^{  \varepsilon_n })
 \big) \\
& = \kern-1em
 \sum_{i_1,\ldots,i_{n}=1}^N  
 \frac{1}{N}
 \E\big( 
 u^{(\varepsilon_1)}_{\sigma_1\circ
 \varepsilon_1(i_1,i_2)} 
 \cdots
  u^{(\varepsilon_{n})}_{\sigma_{n}
 \circ
 \varepsilon_{n}(i_{n},i_1)}
 \big).
 \end{align*}
  where
  $(k_s,l_s)=
  \sigma_s
  \circ
  \varepsilon_s(i_s,j_s)$
   for
    $s=1,2,\ldots,n$.
 
   Denote by
   $ P_2^{ \varepsilon} ( n )   $
   the set of all pairings 
   $ p $ on 
   $ [ n ] $
   such that
    $ \varepsilon_{s} \neq \varepsilon_{p(s)} $
    and define
     $ (k_s, l_s)_{1 \leq s \leq n } $
    via 
     $(k_s,l_s)=
      \sigma_s
      \circ
      \varepsilon_s(i_s,i_{s +1})$,
 with the convention 
 $ i_{ n + 1} = i_1 $.
 Applying (\ref{eq:weingarten_symmetric}), we then obtain
     \begin{align*}
    \E \circ \tr 
    \big( ( U^{ \sigma_1})^{ \varepsilon_1} &
   ( U^{ \sigma_2})^{ \varepsilon_2}
   \cdots 
   ( U^{ \sigma_n})^{  \varepsilon_n }
    \big)
    =
    \sum_{p,q\in \mathcal{P}_2^\varepsilon(n)}
    \VN
    \end{align*}
 where
 $ \overrightarrow{\sigma} = ( \sigma_1,  \sigma_2, \dots, \sigma_{n}) $,
 and
   \begin{align}\label{eq:v:w:n}
   \VN  = \Wg_N(p, q)
   \cdot\frac{1}{N}
   | \AN|,
   \end{align}
   and 
   $ \AN $  
   is  the set 
   \begin{align*}
   \big\{
   (i_1, j_1, \dots,& i_{n}, j_{n}) \in [ N ]^{2n}: 
   \ i_1= j_{n}, i_{s+1} = j_s
   \textrm{ for each } s \in [ {n} -1], \\
    &   \textrm { and }
     k_s = k_{p(s)}, l_s = l_{ q ( s ) } , \textrm { for each } s \in [ {n} ] \big\}.
   \end{align*} 
   For 
    $ S $
     a subset of 
      $ [ m ]$,
       we denote
    \begin{align*}
    \mathcal{F}_{ \overrightarrow{\sigma}, \varepsilon, N }^{ ( p, q)} (S)
    =|\big\{
    (i_s,j_s)_{s\in S}; \mathrm{\ there\ exists\ } & (i_{r},j_{r})_{r\notin S} \mbox{ such that }\\
    & (i_1,j_1,\ldots,i_{m},j_{m})\in \AN
    \big\}
    |.
    \end{align*} 
    We showed in \cite{mpsz1} the result below.
    
    \begin{lemma}\label{lemma:3.2.4}
   If there exists some 
   $ S \subseteq [ n ] $ 
   interval, i. e. 
   		  $ S =\{t+1, t+2, \dots, t+r\} $  
 for some
      $ r>0$,  		  
  such that 
   		  \begin{align*}
   	\mathcal{F}_{ \overrightarrow{\sigma}, \varepsilon, N}^{ ( p, q)} (S) = o\big( N^{  | S | + 1  -|\{B:\  B  \textrm{ block in } p\vee q   \textrm{ and }
   		  			B \subseteq S  \}|} \big) 
   		  \end{align*}
   		  then
   		  $ \displaystyle \VN = o(N^0) = o(1) $.  
    \end{lemma} 
 Moreover, as shown in \cite{mpsz1} (Proposition 3.1),   
  $ \VN $
  has the following properties:
  \begin{enumerate}
  \item[($\mathfrak{v}1$)] $ \VN = O(N^0) = O(1) $
  for any
   $ p, q \in P_2^{\varepsilon}(n) $.
   \item[($\mathfrak{v}2$)] If
   $ p \vee q $ 
   is crossing, then
   $ \VN = O(N^{-1})$.
  \item[($\mathfrak{v}3$)] Suppose that 
  $ B =  \{a_1, a_2, \dots, a_r  \} $ is a block of 
  $ p \vee q $,
  with
  $ a_1 < a_2 < \cdots < a_r $. 
  Then either 
  $ \VN = O(N^{-1})$ 
  or $ r $ is even, for each 
  $ s \in [ r ] $
  we have that
  $ \varepsilon_{a_s} \neq \varepsilon_{a_{s+1}} $
  and
  $ \{ p(a_s), q(a_s) \} = \{ a_{  s - 1 }, a_{ s + 1 } \}$.
  
  Denoting	$ \widetilde{p_{k}}$ 
      	and
      	$ \widetilde{q_{k}} $
      	are pairings on 
      	$ [ 2k ]$ 
      	given by  
      	$ \widetilde{p_{k}} (2l) = 2l + 1 ( \textrm{mod }  2k )$
      	respectively
      	$  \widetilde{q_{k}} ( 2 l ) =
      	 2l -1 ( \textrm{mod } 2k ) $,
  the last condition is equivalent to the pair
  $( p_{| B}, q_{| B}) $
   of restrictions to $ B $  of $ p $,
  respectively $ q $, is either
  	$( \widetilde{p_{\frac{r}{2}}}, \widetilde{q_{\frac{r}{2}}}) $
  	or
  	$(  \widetilde{q_{\frac{r}{2}}},
  	\widetilde{p_{\frac{r}{2}}})$.
\item[($\mathfrak{v}4$)] If
$ \sigma_1 = \sigma_2 = \cdots = \sigma_n = \textrm{Id} $, 
then 
$ U $ is $ R $-diagonal for each $ N $ and the free cumulants are given by
  	\begin{align*}
  	\kappa_{2r}(U_N, U_N^\ast, \dots, U_N, U_N^\ast)
  	 = (-1)^{r-1} \mathit{Cat}_{ r -1}
  &	=\lim_{ N \rightarrow \infty}
   \mathcal{V}_{\overrightarrow{\sigma}, \varepsilon^\prime, N} 
  ( \widetilde{p_{r}}, \widetilde{q_{r}} )
 \\
  \kappa_{2r}( U_N^\ast, U_N, \dots, U^\ast_N, U_N)
  = (-1)^{r-1} \mathit{Cat}_{ r -1}
  &=\lim_{N \rightarrow \infty }
   \mathcal{V}_{ \overrightarrow{\sigma}, \varepsilon^{\prime\prime}, N} 
    ( \widetilde{q_{r}},  \widetilde{p_{r}} )
  \end{align*}
where 
$ \varepsilon^{\prime}, \varepsilon^{\prime\prime}:[2r] \rightarrow \{ 1, \ast \} $
are given by
$ \varepsilon^\prime( 2s ) = \varepsilon^{\prime\prime}(2s-1) = \ast $
and
$ \varepsilon^\prime( 2s - 1 ) = \varepsilon^{\prime\prime}(2s) = 1 $
for each
$ s\in [ r ] $.
   \end{enumerate}
   
\begin{remark}\label{pi:alt}
Denote by 
 $ NC_{\varepsilon, alt} ( n ) $
  the (possibly void)  set of partitions 
 $ \pi $
 on 
 $ [ n ] $
 such that
 \begin{itemize}
 \item[$\bullet$] $ \pi $  is non-crossing 
 \item[$\bullet$] if 
   $\{ a_1, a_2, \dots, a_r \}$
     is a block of 
    $ \pi $ 
    with
     	$ a_1 < a_2 < \dots < a_r $,
  then
   $ r $
   is even and
  $ \varepsilon_1 \neq 
  \varepsilon_2 \neq \dots\neq \varepsilon_r $.
 \end{itemize}
 An immediate consequence of the properties
 ($\mathfrak{v}1$) -- ($\mathfrak{v}3$) 
 from above is the following.
 \begin{align*}
 \E \circ \tr 
     \big( ( U^{ \sigma_1})^{ \varepsilon_1} &
   ( U^{ \sigma_2})^{ \varepsilon_2}
   \cdots 
   ( U^{ \sigma_n})^{  \varepsilon_n }
     \big)
    = \sum_{ \pi \in NC_{\varepsilon, alt}(n)} 
  V_{ \overrightarrow{\sigma}, \varepsilon, N} (\pi) + O(N^{-1}),
 \end{align*}
 where
 $ \displaystyle V_{ \overrightarrow{\sigma}, \varepsilon, N} ( \pi) 
 = \sum_{ \substack{p, q \in P_2^{\varepsilon}(n)\\ p \vee q = \pi }}
  \VN $.
\end{remark}   
 
 \subsection{Main results}

  Within this section, for $ p$ a positive integer, 
  $ U_p $
  will denote a 
  $ p \times p $ 
  Haar unitary random matrix.
  
  Suppose now that 
   $ r $ 
   is a positive integer
   and, for each
   $ i \in [ r ]$, 
   we have that 
    $ \vartheta_i \in \{ -1, 1\} $
  and
  $ ( b_{i, N})_N $,
  $  ( d_{i, N})_N $
  are two non-decreasing sequences of positive integers such that for every
    $ i, i^\prime \in [ r ] $,
  \[
   b_{i, N} d_{i, N} = b_{i^\prime, N } d_{i^\prime, N } = M_N
   \]
 for some strictly increasing sequence
   $ (M_N)_N $.
   
   To simplify the notations, for the rest of the section we will omit the subscript 
   $ N $,
    i.e. we shall write
    $ M, b_s, d_s, \sigma_s $
     for 
   $M_N, b_{s, N}, d_{s, N},  \sigma_{s, N} $.
  
  \begin{proposition}\label{thm:24}
   If
   $ s, t \in [ r ] $ 
   are such that
   \begin{align*}
   \liminf_{N \rightarrow \infty }
   \frac{1}{M^2}
   | \big\{  (i, j) \in [ M ]^2: \ 
   ( \Gamma_{b_s, d_s}^{ (\vartheta_s)})^{ -1}
   \circ
   \Gamma_{b_t, d_t}^{ (\vartheta_t )}
   (i, j) = (i, j)
    \big\} | > 0
   \end{align*}
   then
   $ U_M^{\Gamma_{b_s, d_s}^{ (\vartheta_s)}} $
   and
   $U_M^{ \Gamma_{b_t, d_t}^{ (\vartheta_t)} } $
   are \textbf{not} asymptotically free.
   \end{proposition}  
     
     \begin{proof}
     Denote
      $ \tau_s = \Gamma_{b_s, d_s}^{ (\vartheta_s)} $
      and
      $ \tau_t = \Gamma_{b_t, d_t}^{ (\vartheta_t)}  $.
      
     It suffices to show that
     $ \displaystyle
     \lim_{N \rightarrow \infty} 
     \E \circ \tr
     \big( U_M^{\tau_s}
     ( U_M^{ \tau_t } )^\ast
     \big) \neq 0 $.
     And indeed
     \begin{align*}
     \lim_{N \rightarrow \infty} 
        \E \circ \tr
        \big( U_M^{\tau_s} 
        ( U_M^{ \tau_t } )^\ast &
        \big) 
        \geq
       \lim\inf_{ N \rightarrow \infty}
        \sum_{i_1, j =1}^M 
        \E \big(
        u_{ \tau_1 (i, j)} 
        \overline{ u_{ \tau_2 (i, j)}} 
        \big)\\
        = &
  \liminf_{N \rightarrow \infty}
    \sum_{i_1, j =1}^M \frac{1}{M^2} \delta_{ \tau_1(i, j)}^{ \tau_2(i, j)} >0 .   
     \end{align*}
     \end{proof}

  \begin{theorem}\label{thm:25}
 Denote
  $ \tau_s = \Gamma_{ b_s, d_s}^{( \vartheta_s )}$
  for
  $ s =1, 2, \dots, n $.
  Then the family
  $ U_M^{ \tau_1}, \dots, U_M^{ \tau_n } $
  is asymptotically free if and only if  for any
  $ s \neq t $
  we have that
  $ \Gamma_{b_s, d_s}^{( \vartheta_s)} $
  and
  $ \Gamma_{b_t, d_t}^{(\vartheta_t )} $
  satisfy the condition
\begin{equation}
  \lim_{N \rightarrow \infty }
  \frac{1}{M_N^2}
         | \big\{  (i, j) \in [ M_N ]^2 \mid  
\Gamma_{b_N, d_N}^{ (\vartheta)}(i, j) =
           \Gamma_{b^\prime_N, d^\prime_N}^{ (\vartheta^\prime)}
           (i, j) 
            \big\} |  = 0. \tag{\ref{condi:19}}
\end{equation}
\end{theorem}

\begin{proof}

Let 
$  m $ 
be a positive integer, and, for each
$ k \in [ m ] $,
let
$ \varepsilon_k \in \{ 1, \ast \} $
and let 
$ \big( \sigma_{k, N} \big)_N  \in  \big\{ 
\big(  \Gamma_{b_{ s, N}, d_{ s, N } }^{( \vartheta_s)}   \big)_N : s \in [ r ] \big\} $.
With the notations from above, it suffices to show that in the expansion of
\[
   \E \circ \tr 
   \big( ( U_{M_N}^{ \sigma_{1, N}})^{ \varepsilon_1} 
  ( U_{M_N}^{ \sigma_{2, N}})^{ \varepsilon_2}
  \cdots 
  ( U_{M_N}^{ \sigma_{m, N}})^{  \varepsilon_m }
   \big)
 \]
all mixed free cumulants cancel asymptotically.

Fix 
$ \pi \in NC_{\varepsilon, \textrm{alt}} ( m ) $.
  Since 
  $ \pi $ 
  is non-crossing, it has a block which is also a segment, say
  $ ( t + 1, t+2, \dots, t + 2r )$. 
  Via a circular permutation, we can further suppose that 
  $ t =0 $,
  that is $ (1, 2, \dots, 2r) $ is a block in 
  $ \pi $.
   Denote by
  $ \varepsilon^\prime$,
  respectively 
  $ \varepsilon^{ \prime \prime } $     
   the restrictions of 
  $ \varepsilon $
  to the sets
  $ [ 2 r ] $,
  respectively 
  $ [ m ] \setminus [ 2r ] $.
  Similarly, denote
   $ \overrightarrow{\sigma}^\prime = 
   ( \sigma_1, \dots, \sigma_{ 2 r } ) $
  and
  $ \overrightarrow{\sigma}^{ \prime \prime } =      
       ( \sigma_{2r+1}, \sigma_{2r+2}, \dots, \sigma_m )$.
 With these notations,  
$ \pi = [ 2 r ] \oplus \pi^{\prime \prime }$,
for some 
$ \pi^{ \prime \prime } \in NC_{\varepsilon^{ \prime \prime }, \textrm{alt}} ( m  - 2 r ) $. 
Henceforth, for the conclusion to follow from Remark \ref{pi:alt}, it suffices to show that
 \begin{align}\label{v:k:v}
V_{ \overrightarrow{\sigma}, \varepsilon, M} (\pi)  =  
\mathcal{K}(r, \overrightarrow{\sigma}^\prime, \eta^{\prime}) 
 \cdot
V_{ \overrightarrow{\sigma}^{ \prime \prime }, \varepsilon^{ \prime \prime }, M} ( \pi^{ \prime \prime } )
  +  o(M^0), 
  \end{align}
where
\begin{align*}
\mathcal{K}(r, \overrightarrow{\sigma}^\prime, \eta^{\prime})=
\left\{
\begin{array}{l}  \displaystyle \lim_{ N \rightarrow \infty } 
\kappa_{2r}  \big( 
( U_M^{\sigma_1} )^{\varepsilon_1} ,
\dots, 
( U_M^{\sigma_{1} } )^{\varepsilon_{2r}}
\big), \textrm{ if } \sigma_1 = \dots = \sigma_{2r}  \\
\\
0, \textrm{ otherwise.}
\end{array}
\right.
\end{align*}
 
 We can also assume that 
 $ \varepsilon_{ t + s } \neq \varepsilon_{ t + s + 1} $
for all
$ s \in \{ 1, \dots, 2r -1 \}$.
Otherwise, as 
$ N \rightarrow \infty $,
 the right-hand side of (\ref{v:k:v}) vanishes 
  according to property ( $\mathfrak{v}3$) while the 
left-hand side vanished according to Theorem \ref{thm:16} and property 
($\mathfrak{v}4$), so the equality holds true.

Let
$p, q \in P_2^{\varepsilon}(m) $
be such that
$ p \vee q = \pi $.
 If 
  $ \sigma_{ s } \neq \sigma_{ s + 1 } $
for some 
   $ s \in \{ 1, 2, \dots, 2r-1\} $,
then  property ($\mathfrak{v}3$) gives that either  
$ \displaystyle \lim_{N \rightarrow \infty }
    V_{ \overrightarrow{\sigma}, \varepsilon, M} (\pi) = 0$
(in particular (\ref{v:k:v}) holds true), or    
$s+1 \in \{ p(s), q(s) \} $. 
Suppose first that 
 $ s + 1 = p(s)$.
 If 
 $ r =1  $, 
then condition (\ref{condi:19}) gives that
    \begin{multline*}
    \mathcal{F}_{\overrightarrow{\sigma}, \varepsilon, M }^{(p, q)} \big( \{ s,  s + 1 \} \big) \\
      \leq 
      |
        \big\{ (i_1, i_2, j) \in [ M ]^3 \mid
        \tau_1 (i_1, j) = \tau_2 ( i_2, j) 
         \big\}
          |
        = o(M^2).
    \end{multline*}
    If 
    $ r > 1 $,
     then again (\ref{condi:19}) gives that 
     \begin{align*}
  & \mathcal{F}_{\overrightarrow{\sigma}, \varepsilon, M }^{(p, q)} 
  \big( \{ s,  s + 1 \} \big) 
  \leq
  |
  \big\{ (i_1, j_1, i_2, j_2) \in [ M ]^4 :\ 
  i_2 = j_1 \textrm{ and }\\
  & \hspace{4cm}
   \pi_1 \circ \tau_1 ( i_1, j_1) =
  \pi_1 \circ \tau_2 (j_2, i_2) 
  \big\}
  |\\
  & \leq
  |
  \big\{
  (i_1, i_2, j) \in [ M ]^3:\
  \pi_1 \circ \tau_1 ( i_1, j) =
  \pi_1 \circ \tau_2 (i_2, j)
  \big\}
  | 
  =o(M^2).
     \end{align*}
 In both situations we have that
 $
 \displaystyle
 \lim_{N \rightarrow\infty}
    \mathcal{V}_{ \overrightarrow{\sigma}, \varepsilon, M} ( p, q ) =0,
 $  
 according to Lemma \ref{lemma:3.2.4}.
 The argument for the case 
 $ q ( s ) = s + 1  $
 is similar.

   The rest of the proof, that is the case 
$ \sigma_1 = \dots = \sigma_{2r} $ ,   
    follows the argument for Example 5.3 from \cite{mpsz1}, as shown below.

Let
      $ \Gamma_{b, d}^{( \vartheta)} $
      the common value of 
            $ \sigma_1, \sigma_2,\dots, \sigma_{2r} $
 (here we use again the convention of omitting the index 
 $ N $, 
 that is 
 $ \sigma_j$'s, $ b $ and $ d $
  are depending on $ N $).
 since 
 $ \Gamma_{b, d}^{(-1)} = \big( \Gamma_{b, d}^{ ( 1 )} \big)^t $,
 it suffices to show relation (\ref{v:k:v}) for 
 $ \vartheta = 1 $
 and the case
 $ \vartheta = -1 $
 follows by taking transposes.
 
For 
            each
            $ s \in [ 2 r ] $,
            write
            \begin{align*}
            i_s = ( \alpha_s -1) d + \beta_s 
            \end{align*}
            with
            $ \alpha_s \in [ b ] $
            and 
            $ \beta_s \in [ d ] $.
Assuming that 
$ \vartheta = 1 $,
suppose first that 
$ \eta_1 = 1 $;
then, for
$ s \in [ 2 r ] $,   
 \begin{align*}
 (k_s, l_s) = \left\{
 \begin{array}{ll}
  \big( ( \alpha_s -1 ) d + \beta_{s + 1} , ( \alpha_{ s + 1} - 1 )d + \beta_s \big)&
  \textrm{ if $ s $ is odd }\\
  \big( (\alpha_{ s + 1} - 1 ) d + \beta_s ), ( \alpha_s -1)d + \beta_{s+1}
  \big) & 
\textrm{ if $ s $ is even }.
\end{array}
\right.
\end{align*} 
 
 From  property ($\mathfrak{v}3$),  we have that 
 $ \displaystyle \lim_{ N \rightarrow \infty} 
\mathcal{V}_{ \overrightarrow{\sigma}, \varepsilon, M} ( p, q ) =0$
 unless the pair
$  \big( p^{\prime},q^{\prime}\big) $
is either
 $ \big( \widetilde{p_{r}} , \widetilde{q_r} \big) $
       or
$ \big(  \widetilde{q_r}, \widetilde{p_{r}}  \big) $. 
  If      
      $ p^{ \prime } = \widetilde{p_{r}} $ 
    and
    $ q^{ \prime } = \widetilde{q_{r}}$,
    then the conditions
    $ k_s = k_{ p(s) } $ 
    and 
    $ l_s = l_{ q (s ) } $
    become
     \begin{align}\label{cond:0:1}
  \left\{
  \begin{array}{l l}
  \alpha_1 = \alpha_{ 2r + 1} \\
  \beta_{s} = \beta_{ s + 2 } \textrm{ for each } s = 1, 2, \dots,  2 r -1, 
  \end{array} 
  \right.
  \end{align}
while if
  $ p^{\prime} = \widetilde{q_r } $
 and
 $ q^{ \prime  }  = \widetilde{ p_r } $
 then the conditions
 $ k_s = k_{ p(s) } $ 
 and 
 $ l_s = l_{ q (s ) } $
 become
     \begin{align}\label{cond:0:2}
 \left\{
 \begin{array}{l l}
 \beta_1 = \beta_{ 2r + 1} \\
 \alpha_{s} = \alpha_{ s + 2 } \textrm{ for each } s = 1, 2, \dots,  2 r -1.
 \end{array} 
 \right.
 \end{align} 
 
  Either way, we obtain that 
 $\alpha_1 = \alpha_{2r + 1} $ 
 and
 $ \beta_1 = \beta_{2r + 1} $, 
 that is
 $ i_1 = i_{2r + 1 } $, 
 so property ($\mathfrak{v}1$) and the multiplicativity of the leading term in the development of  the unitary Weingarten function in (\ref{eq:weingarten_order})  gives that  
 \begin{align*}
 V_{ \overrightarrow{\sigma}, \varepsilon, M} (\pi)  =  
 V_{ \overrightarrow{\sigma}^{ \prime}, \varepsilon^{ \prime }, M}
  (\pi^{ \prime })  
 \cdot
 V_{ \overrightarrow{\sigma}^{ \prime \prime}, \varepsilon^{ \prime \prime}, M} (\pi^{ \prime \prime} )  
 + o(M^0), 
 \end{align*}       
 so it suffices to show that
\begin{align}\label{k:v:v}
 \lim_{ N \rightarrow \infty } 
\kappa_{2r}  \big( 
( U_M^{\Gamma_{b, d}^{(1)} } )^{\varepsilon_1} , &
( U_M^{\Gamma_{b, d}^{(1)} } )^{\varepsilon_2},
\dots, 
( U_M^{\Gamma_{b, d}^{(1)} } )^{\varepsilon_{2r}}
\big)
= \lim_{ N \rightarrow \infty } V_{ \overrightarrow{\sigma}^{ \prime}, \varepsilon^{ \prime }, M}
  (\pi^{ \prime }) \\
 & = \sum_{ ( p^\prime, q^\prime ) } \lim_{ N \rightarrow \infty} 
   \mathcal{V}_{ \overrightarrow{\sigma}^{ \prime}, \varepsilon^{ \prime }, M}
    (p^{ \prime }, q^\prime ) \nonumber
\end{align} 
where the last summation is done over 
$ (p^\prime, q^\prime ) \in \big\{
 ( \widetilde{p_{r}} , \widetilde{q_r} )
 ,
(  \widetilde{q_r}, \widetilde{p_{r}}  ) 
\big\} $.

For 
$ r =1 $, 
we have that 
$ \widetilde{p_{r}} = \widetilde{q_r} = (1, 2) $
so the summation in the right-hand side of (\ref{k:v:v}) has just one term which, according to property ($\mathfrak{v}4$) equals 
$1 $.
On the other hand, Theorem  \ref{thm:16} of Section 3 gives that 
the left-hand side of (\ref{k:v:v}) equals the left-hand side of the first relation of property ($\mathfrak{v}4$), that is also 
$1$.

For 
$ r > 1 $, 
applying again Theorem  \ref{thm:16} of Section 3 and property 
($\mathfrak{v}4$), we have that the left-hand side of (\ref{k:v:v}) is given by
\begin{multline*}
\lim_{ N \rightarrow \infty} \kappa_{2r} \big( 
( U_M^{\Gamma_{b, d}^{(1)} } )^{\varepsilon_1} ,
 ( U_M^{\Gamma_{b, d}^{(1)} } )^{\varepsilon_2},
\dots, 
( U_M^{\Gamma_{b, d}^{(1)} } )^{\varepsilon_{2r}}
	  \big) \\
\mbox{} =    
 \begin{cases}
 0 & \textrm{ if }  b\rightarrow \infty \textrm{ and } d \rightarrow \infty \\
 b^{2-2r} (-1)^{r-1} \mathit{Cat}_{r-1}  & \textrm{ if }  b \textrm{ is bounded } \\
 d^{2-2r} (-1)^{ r -1} \mathit{Cat}_{r-1}  &
 \textrm{ if }   d \textrm{ is bounded. }
 \end{cases} 
\end{multline*}

On the other hand, the summation in the right hand side of (\ref{k:v:v}) has now tow terms, one for 
$( p^\prime , q^\prime ) = ( \widetilde{p_r}, \widetilde{q_r} ) $
and one for
$( p^\prime , q^\prime ) = (  \widetilde{q_r}, \widetilde{p_r} ) $.
Let us analyze the case
$ ( p^\prime, q^\prime) = ( \widetilde{ p_r}, \widetilde{ q_r }) $.
 Condition (\ref{cond:0:1}) gives that
   \[
      \mathcal{F}_{\overrightarrow{\sigma}, \varepsilon, M}^{ ( p, q )} ( [ 2 r ]) \leq b^{2r} d^{ r} = \frac{ M^{2r}}{ d^r}.
    \]
So, if
$ d \rightarrow \infty $,
Lemma \ref{lemma:3.2.4} gives that 
$ \displaystyle 
\lim_{ N \rightarrow \infty } \mathcal{V}_{ \overrightarrow{\sigma}, \varepsilon, M} ( p, q) = 0 $.    
If
$ d $
is bounded, since 
$ 2r $ 
is even, 
(\ref{cond:0:1}) 
gives that 
$ \alpha_1 = \alpha_{2 r +1} $,
therefore
\begin{align*}
i_1 = ( \alpha_1 -1)d + \beta_1 = 
( \alpha_{ 2r+1} -1) d + \beta_{2 r +1 } = i_{ 2r + 1}.
\end{align*}  

In particular, each 
 $ 2 (m -2r)$-tuple
 $ (i_s, j_s)_{ 2r+1 \leq s \leq m } $
 from $
 \mathcal{A}_{\overrightarrow{\sigma}\prime, \varepsilon^\prime, M}^{ ( p^\prime, q^\prime )} $ uniquely determines  $ \alpha_1 $
 and
 $ \beta_1 $  
 (via
 $ i_{ 2r +1 } $
 ).  Hence we have that
 
 \begin{align*}
 | \mathcal{A}_{\overrightarrow{\sigma}, \varepsilon, M}^{ (p, q)}| = &
 |\{ ( \alpha_s, \beta_s)_{ 2\leq s \leq m } : \ \alpha_s, \beta_s \textrm{ satisfy conditions (\ref{cond:0:1})}\}  | 
 \cdot
 | \mathcal{A}_{\overrightarrow{\sigma}^\prime, \varepsilon^\prime, M}^{ ( p^\prime, q^\prime )}|\\
 = &
  b^{ 2r -1} d \cdot
 | \mathcal{A}_{\overrightarrow{\sigma}^\prime, \varepsilon^\prime, M}^{ ( p^\prime, q^\prime )}|.
 \end{align*}
 
 Remember that
 $ \pi = [ 2 r ] \oplus \pi^{ \prime \prime } $, 
 so
 $ | \pi | = 1 + | \pi ^{ \prime \prime } | $.
In particular the first formula from Notation \ref{not:21} becomes
     \begin{align}\label{eq:c:p:q}
     w_1(\pi) = (-1)^{r -1} \mathit{Cat}_{ r -1} 
     \cdot
     w_1( \pi^\prime ). 
     \end{align}
So, using (\ref{eq:c:p:q}) and (\ref{eq:v:w:n}), we get
   \begin{eqnarray*}
 	\lefteqn{\mathcal{V}_{
 		\overrightarrow{\sigma},\varepsilon, M}
 	(p, q) }\\
 	&=&  
 	\Wg_{ M} (p, q) \cdot \frac{1}{M}    | \mathcal{A}_{\overrightarrow{\sigma}, \varepsilon, M}^{ (p, q)}| \\
 	& = &
 	( M )^{ - m + | \pi | } \cdot
 	w_1( \pi )  \frac{1}{M}   | \mathcal{A}_{\overrightarrow{\sigma}, \varepsilon, M}^{ (p, q)}|
 	+ O ( M^{ -2} ) \\
 	& = &
 	(-1)^{ r -1} \mathit{Cat}_{ r -1} 
 	\cdot
 	w_1( \pi^\prime) \cdot 
 	( M)^{ - 2r + 1 } \cdot
 	( M )^{ - ( m - 2r ) +
 		|  \pi^\prime | } 
 	\frac{1}{M} \\
 	&& \mbox{} \times d b^{ 2r -1} 
 	| \mathcal{A}_{\overrightarrow{\sigma}^\prime, \varepsilon^\prime, M }^{ ( p^\prime, q^\prime ) }  | 
 	+ O( M^{ -2} )\\
 	& = &
 	\big[ d^{ 2- 2r} \cdot
 	(-1)^{ r -1} 
 	\mathit{Cat}_{ r -1} 
 	\big]
 	\cdot
 	\big[ w_1 ( \pi^\prime)
 	( M )^{ - ( m - 2r ) 
 		+ | \pi^\prime | } \\
 	&& \mbox{} \times \frac{1}{M}
 	| \mathcal{A}_{\overrightarrow{\sigma}^\prime, \varepsilon^\prime, M }^{
 		( p^\prime, q^\prime )
 	}  |
 	\big] + O(M^{ -2 } )\\
 	& = &
 	\big[ d^{ 2- 2r } 
 	\cdot
 	(-1)^{ r -1} 
 	\mathit{Cat}_{ r -1} 
 	\big]
 	\cdot
 	\mathcal{V}_{ \overrightarrow{\sigma}^\prime, \varepsilon^\prime, M} ( p^\prime, q^\prime )
 	+ O(M^{ -2 } ).
 \end{eqnarray*}   
 
 In the case
  $ ( p^\prime, q_{ | [ 2 r ]}) = ( \widetilde{ q_r },  \widetilde{ p_r } ) $,
  the conditions (\ref{cond:0:2}) give that 
  \[
 \mathcal{F}_{\overrightarrow{\sigma}, \varepsilon, M}^{ ( p, q )} ( [ 2 r ])
 \ab\leq 
 \frac{M^{2r}}{b^r} 
 \mbox{\ and\ }
| \mathcal{A}_{\overrightarrow{\sigma}, \varepsilon, M}^{ (p, q)}| =
   b d^{2r-1}  \cdot 
 | \mathcal{A}_{\overrightarrow{\sigma}^\prime, \varepsilon^\prime, M}^{ ( p^\prime, q^\prime )}|
 \]  
 therefore, using again (\ref{eq:v:w:n}) and (\ref{eq:c:p:q}) we obtain that  
  \begin{eqnarray*}\lefteqn{
  \mathcal{V}_{ \overrightarrow{\sigma}, \varepsilon, M} (p, q)}\\
   && \mbox{} =
  \begin{cases}
  o(M^0) & \textrm{ if } b \rightarrow \infty \\
  \big[ b^{ 2- 2r } 
  \cdot
  (-1)^{ r -1} 
  \mathit{Cat}_{ r -1} 
  \big]
  \cdot
  \mathcal{V}_{ \overrightarrow{\sigma}^\prime, \varepsilon^\prime, M} ( p^\prime, q^\prime ) &  \\
 \hskip 6em\mbox{} + O(M^{ -2 } ) & \textrm{ if  $ b $ is bounded. }
  \end{cases}
  \end{eqnarray*}

 The case 
 $ \varepsilon_1 = \ast $
  is similar, interchanging the symbols
   $ \widetilde{ p_r} $
   and 
   $ \widetilde{ q_r } $.  
\end{proof}   

We conclude this section with some immediate consequences of Theorem \ref{thm:24}. First, taking
 $ n =2 $, 
  $ \vartheta_1 = 1 $, $ \vartheta_2 = -1 $
  and
  $  b_{1, N} = b_{2, N} $, Lemma \ref{lemma:gamma} and Theorem \ref{thm:25} give the following.
  
  \begin{corollary}\label{cor:26}
  	For any 
  	$ \vartheta, b, d $,
  	we have that
  	$ U_M^{ \Gamma_{b, d}^{(\vartheta)}}$
  	is asymptotically free from its transpose.
  \end{corollary}

   In \cite{mp2} it is proved that unitarily invariant random matrices are asymptotically free from their transposes. An example of random matrices which are not unitarily invariant but are free from their transposes is given by Wigner ensembles (see \cite{m}). Henceforth, Corollary \ref{cor:26}, together with Remark \ref{non-unit-inv} gives a new non-trivial class of non-unitarily invariant random matrices which are free from their transposes.
  
 Furthermore, if in Theorem \ref{thm:25} we take
  $ n = 4 $,  $ b_{1, N} = b_{2, N} = 1 $, $ \vartheta_1 = \vartheta_3 = 1 $,
  $ \vartheta_2 = \vartheta_4 = -1 $
  and $ b_{3, N} = b_{4, N} $ such that 
  $ b_{ 3, N}, d_{3, N} \rightarrow \infty $,
  another application of Lemma \ref{lemma:gamma} gives the result below.

\begin{corollary}\label{cor:27}
If 
$ b \rightarrow \infty $
and
$ d \rightarrow \infty$
then
$ U_M,  U_M^{ \Gamma_{b, d}^{(\vartheta)}}$
and their transposes
form an asymptotically free family.
\end{corollary}

\section*{acknowledgments}
We are grateful to Adam Skalski for bringing \cite{cu} to our attention.

\thebottomline

\end{document}